%% file: faster-fista.tex
\documentclass[11pt,reqno,twoside]{article}
\usepackage{mysty}

\usepackage{wrapfig}

\usepackage{fullpage}
\usepackage[top=2.35cm, bottom=2.35cm, left=2.35cm, right=2.35cm]{geometry}
 
\usepackage[noindentafter]{titlesec}
\usepackage{bold-extra}

\titlespacing*{\section}{0pt}{1.25\baselineskip}{0.25\baselineskip}
\titlespacing*{\subsection}{0pt}{0.75\baselineskip}{0.125\baselineskip}
\titlespacing*{\subsubsection}{0pt}{0.5\baselineskip}{0.125\baselineskip}
\titlespacing*{\paragraph}{0pt}{0.25\baselineskip}{0.25\baselineskip}

\makeatletter
\g@addto@macro\normalsize{%
  \setlength\abovedisplayskip{5pt}
  \setlength\belowdisplayskip{5pt}
  \setlength\abovedisplayshortskip{5pt}
  \setlength\belowdisplayshortskip{5pt}
}
\makeatother

\usepackage{patchcmd}
\makeatletter
\patchcommand\@starttoc{\begin{quote}}{\end{quote}}
\makeatother

\usepackage{enumitem}

\SetLabelAlign{parright}{\parbox[t]{\labelwidth}{\raggedleft{#1}}}
\setlist[description]{style=multiline,topsep=4pt,align=parright}

\makeatletter
\let\reftagform@=\tagform@
\def\tagform@#1{\maketag@@@{(\ignorespaces\textcolor{black}{#1}\unskip\@@italiccorr)}}
\newcommand{\iref}[1]{\textup{\reftagform@{\tcr{\ref{#1}}}}}
\makeatother

\makeatletter
\newcommand{\pushright}[1]{\ifmeasuring@#1\else\omit\hfill$\displaystyle#1$\fi\ignorespaces}
\newcommand{\pushleft}[1]{\ifmeasuring@#1\else\omit$\displaystyle#1$\hfill\fi\ignorespaces}
\makeatother

\begin{document}
	

\title{Improving ``Fast Iterative Shrinkage-Thresholding Algorithm'': Faster, Smarter and Greedier}
\author{Jingwei Liang\thanks{School of Mathematical Sciences, Queen Mary University of London, London UK. E-mail: jl993@cam.ac.uk.}
		\and
		Tao Luo\thanks{School of Mathematical Sciences, Shanghai Jiao Tong University, Shanghai China. E-mail: luotao41@sjtu.edu.cn.}
		\and
		Carola-Bibiane Sch{\"{o}}nlieb\thanks{DAMTP, University of Cambridge, Cambridge UK. E-mail: cbs31@cam.ac.uk.}}
\date{}
\maketitle

\begin{abstract}
The ``fast iterative shrinkage-thresholding algorithm'', a.k.a. FISTA, is one of the most well-known first-order optimization scheme in the literature, as it achieves the worst-case $O(1/k^2)$ optimal convergence rate in terms of objective function value. 
However, despite such an optimal theoretical convergence rate, in practice the (local) oscillatory behavior of FISTA often damps its efficiency. 
Over the past years, various efforts are made in the literature to improve the practical performance of FISTA, such as monotone FISTA, restarting FISTA and backtracking strategies. 
In this paper, we propose a simple yet effective modification to the original FISTA scheme which has two advantages: it allows us to 1) prove the convergence of generated sequence; 2) design a so-called ``lazy-start'' strategy which can be up to an order faster than the original scheme. 
Moreover, we propose novel adaptive and greedy strategies which probe the limit of the algorithm. 
The advantages of the proposed schemes are tested through problems arising from inverse problem, machine learning and signal/image processing. 
\end{abstract}




\input{tex/sec-introduction}

\input{tex/sec-preliminary}

\input{tex/sec-fista-mod}

\input{tex/sec-lazystart}

\input{tex/sec-adaptive}

\input{tex/sec-nesterov}

\input{tex/sec-experiment}

\section{Conclusions}

We proposed a simple modification to the original FISTA-BT scheme, which allows us to prove the convergence of the sequence generated by the modified scheme. We also proposed a lazy-start strategy which can greatly improve the practical performance of FISTA schemes. Several adaptive schemes were also developed, which can adaptively adjust to the (local) properties of the problem to solve. The performances of the proposed schemes were verified on various problems arising from inverse problems, data science and computer vision.

\subsection*{Acknowledgement}
We would like to thank Dr. Robert Tovey for helpful discussions and comments of the paper. 
JL acknowledges support from the Leverhulme Trust and Newton Trust. 
CBS acknowledges support from the Leverhulme Trust project on Breaking the Non-Convexity Barrier, and on Unveiling the Invisible, the Philip Leverhulme Prize, the EPSRC grant No. EP/S026045/1, EPSRC grant No. EP/M00483X/1, and EPSRC Centre No. EP/N014588/1, the European Union Horizon 2020 research and innovation programmes under the Marie Skłodowska-Curie grant agreement No. 691070 CHiPS and the Marie Skłodowska-Curie grant agreement No 777826, the Cantab Capital Institute for the Mathematics of Information, and the Alan Turing Institute.

\begin{small}
\bibliographystyle{plain}
\bibliography{ifb}
\end{small}

\end{document}

%% file: tex/sec-introduction.tex

\section{Introduction}\label{sec:introduction}

The acceleration of first-order optimization methods is an active research topic of non-smooth optimization. Over the past decades, various acceleration techniques are proposed in the literature. Among them, one most widely used is called ``inertial technique'' which dates back to \cite{polyak1964some} where Polyak proposed the so called ``heavy-ball method'' which dramatically speeds up the practical performance of gradient descent. 
In a similar spirit, in~\cite{nesterov83} Nesterov proposed another accelerated scheme which improves the $O(1/k)$ objective function convergence rate of gradient descent to $O(1/k^2)$. 
The extension of \cite{nesterov83} to the non-smooth case was due to \cite{fista2009} where Beck and Teboulle proposed the FISTA scheme which is the main focus of this paper. 

In this paper, we are interested in the following structured non-smooth optimization problem, which is the sum of two convex functionals, 
\beq\label{eq:min-problem}\tag{$\mathcal{P}$}
\min_{x\in\calH} ~~\Phi(x) \eqdef F(x) + R(x)  ,
\eeq
where $\calH$ is a real Hilbert space. The following assumptions are assumed throughout the paper
\begin{enumerate}[leftmargin=3.75em,label= ({\textbf{H.\arabic{*}})},ref= \textbf{H.\arabic{*}}]
\item \label{A:R} $R : \calH \to ]-\infty, +\infty]$ is proper, convex and lower semi-continuous (lsc);
\item \label{A:F} $F : \calH \to ]-\infty, +\infty[$ is convex and differentiable, with gradient $\nabla F$ being $L$-Lipschitz continuous for some $L > 0$;
\item \label{A:argmin} The set of minimizers is non-empty, \ie $\Argmin(\Phi) \neq \emptyset$.
\end{enumerate}
Problem \eqref{eq:min-problem} covers many problems arising from inverse problems, signal/image processing, statistics and machine learning, to name few. We refer to Section \ref{sec:experiment} the numerical experiment section for concrete examples.

\subsection{Forward--Backward-type splitting schemes}

In the literature, one widely used algorithm for solving \eqref{eq:min-problem} is Forward--Backward splitting (FBS) method~\cite{lions1979splitting}, which is also known as \emph{proximal gradient descent}.

\paragraph{Forward--Backward splitting}

With initial point $x_{0} \in \calH$ chosen arbitrarily, the standard FBS iteration without relaxation reads as
\beq\label{eq:fbs}
\xkp \eqdef \prox_{\gamma_k R} \Pa{\xk - \gamma_k \nabla F(\xk)} ,~~ \gamma_k \in ]0, 2/L]  ,
\eeq
where $\gamma_k$ is the step-size, and $\prox_{\gamma R}$ is called the \emph{proximity operator} of $R$ defined by 
\beq\label{eq:prox}
\prox_{\gamma R} (\cdot) \eqdef \argmin_{x\in\calH} \gamma R(x) + \sfrac{1}{2}\norm{x-\cdot}^2	.
\eeq
%
Similar to gradient descent, FBS is a descent method, that is the objective function value $\Phi(\xk)$ is non-increasing under properly chosen step-size $\gamma_k$. The convergence properties of FBS are well established in the literature, in terms of both sequence and objective function value:
\begin{itemize}
	\item The convergence of the generated sequence $\seq{\xk}$ and the objective function value $\Phi(\xk)$ are guaranteed as long as $\gamma_k$ is chosen such that $0 < \underline{\gamma} \leq \gamma_k \leq \bar{\gamma} < \sfrac{2}{L}$ \cite{combettes2005signal}.  
	\item Convergence rate: we have $\Phi(\xk) - \min_{x\in\calH}\Phi(x) = o(1/k)$ for the objective function value~\cite{molinari2018convergence} and $\norm{\xk-\xkm} = o(1/\sqrt{k})$ for the sequence $\seq{\xk}$~\cite{liang2016convergence}. Moreover, linear convergence rate can be obtained under for instance strong convexity.
\end{itemize}
Over the years, numerous variants of FBS have been proposed under different purposes, below we particularly focus on its inertial accelerated variants.

%
%
%


\paragraph{Inertial Forward--Backward}

The first inertial Forward--Backward was proposed by Moudafi and Oliny in~\cite{moudafi2003convergence}, under the setting of finding zeros of monotone inclusion problems. Specifying the algorithm to the case of solving \eqref{eq:min-problem}, we obtain the following iteration: 
\beq\label{eq:ifb}
\begin{aligned}
	\yk &= \xk + \ak (\xk - \xkm)	,	\\
	\xkp &= \prox_{\gamma_k R} \Pa{\yk - \gamma_k \nabla F(\xk)}  ,~ \gamma_k \in ]0, 2/L[ ,
\end{aligned}
\eeq
where $\ak$ is the \emph{inertial parameter} which controls the momentum $\xk-\xkm$. 
The above scheme recovers the heavy-ball method when $R = 0$ \cite{polyak1987introduction}, and becomes the scheme of \cite{lorenz2015inertial} if we replace $\nabla F(\xk)$ with $\nabla F(\yk)$. We refer to \cite{liang2017activity} for a more general discussion of inertial Forward--Backward splitting schemes.

The convergence of \eqref{eq:ifb} can be guaranteed under proper choices of $\gamma_k$ and $\ak$. Under the same step-size choice, \eqref{eq:ifb} could be significantly faster than FBS in practice. 
However, except for special cases (\eg quadratic problem as in \cite{polyak1987introduction}), in general there is no convergence rate established for \eqref{eq:ifb}.

\paragraph{The original FISTA}

By the form of iteration, FISTA is a particular example of the class of inertial FBS algorithms. What differentiates FISTA from others is the restriction on step-size $\gamma_k$ and special rule for updating $\ak$. Moreover, FISTA schemes have convergence rate guarantee on the objective function value, which is the consequence of $\ak$ updating rule. 
The original FISTA scheme of \cite{fista2009} is described below in Algorithm~\ref{alg:fista}.

\begin{center}
\begin{minipage}{0.925\linewidth}
\begin{algorithm}[H]
\caption{The original FISTA scheme (FISTA-BT)} \label{alg:fista}
{\noindent{\bf{Initial}}}: $t_0 = 1$, $\gamma = 1/L$ and $x_{0} \in \calH, x_{-1} = x_{0}$, $ k = 1$. \\ 
\Repeat{convergence}{ \vspace{-1em}
\beq \label{eq:fista-bt}
\begin{gathered}
\textstyle \tk = \frac{1 + \sqrt{1 + 4\tkm^2}}{2}  , ~~~  \ak = \frac{\tkm-1}{\tk}  ,  \\
\yk = \xk + \ak\pa{\xk-\xkm}    ,  \\
\xkp = \prox_{\gamma R}\Pa{\yk - \gamma \nabla F(\yk)} . 
\end{gathered}
\eeq
$k = k + 1$\;
}
\end{algorithm}
\end{minipage}
\end{center}

As described, FISTA first computes $\tk$ and then updates $\ak$ with $\tk$ and $\tkm$. Due to the choices of parameters, FISTA achieves $O(1/k^2)$ convergence rate for $\Phi(\xk) - \min_{x\in\calH}\Phi(x)$ which is optimal \cite{nemirovsky1983problem}. 
For the rest of the paper, to distinguish the original FISTA from the one in~\cite{chambolle2015convergence} and the proposed modified FISTA scheme, we shall use ``FISTA-BT'' to refer Algorithm~\ref{alg:fista}.

\paragraph{A sequence-convergent FISTA}

Although achieving optimal convergence rate for objective function value, the convergence of the sequence $\seq{\xk}$ generated by Algorithm~\ref{alg:fista} was initially an open problem. 
This question was answered in \cite{chambolle2015convergence}, where Chambolle and Dossal proved the convergence of $\seq{\xk}$ by considering the following rule to update $\tk$: let $d > 2$ and
\beq\label{eq:fista-cd}
\tk = \sfrac{k+d}{d}  ,~~~
\ak = \sfrac{\tkm-1}{\tk} = \sfrac{k-1}{k+d}  .
\eeq
Such a rule maintains the $O(1/k^2)$ objective convergence rate, and also allows the authors to prove the convergence of $\seq{\xk}$. Later on in \cite{AttouchFISTAGrad15}, \eqref{eq:fista-cd} was studied under the continuous time dynamical system setting, and the convergence rate of objective function is proved to be $o(1/k^2)$~\cite{AttouchFISTA15}. 
For the rest of the paper, we shall use ``FISTA-CD'' to refer to \eqref{eq:fista-cd}.

\subsection{Problems}

Although theoretically FISTA-BT achieves the optimal $O(1/k^2)$ convergence rate, in practice it could be even slower than the non-accelerated Forward--Backward splitting scheme, which is mainly caused by the oscillatory behavior of the scheme \cite{liang2017activity}. 
In the literature, several modifications of FISTA-BT are proposed to deal with such oscillation, such as the monotone FISTA \cite{beck2009fast} and restarting FISTA \cite{o2012adaptive}. 
Other work includes FISTA-CD~\cite{chambolle2015convergence} for the convergence of iterates, and a backtracking strategy for adaptive Lipschitz constant estimation~\cite{calatroni2017backtracking}. 
Despite these works, there are still important questions to answer:
\begin{itemize}
	\item 
	Although \cite{chambolle2015convergence} proves the convergence of the iterates $\seq{\xk}$ under $\tk$ updating rule \eqref{eq:fista-cd}, the convergence of $\seq{\xk}$ for the original FISTA-BT remains unclear. 
	\item 
	The practical performance of FISTA-CD is almost identical to FISTA-BT if $d$ of \eqref{eq:fista-cd} is chosen close to $2$. However, when relatively large values of $d$ are chosen, significant practical acceleration can be obtained. For instance, it is reported in \cite{liang2017activity} that for $d = 50$ the resulted performance can be several times faster than $d=2$. However, there is no proper theoretical justifications on how to choose the value of $d$ in practice. 
	\item
	When the problem \eqref{eq:min-problem} is strongly convex, there exists an optimal choice for $\ak$~\cite{nesterov2004introductory}. However, in practice, very often the problem is only locally strongly convex with unknown strong convexity, and estimating the strong convexity could be time consuming. This leads to the question of whether there is a low-complexity approach to estimate strong convexity, or do we really need a tight estimation of it?
	\item 
	Restarting FISTA successfully suppresses the oscillatory behavior of FISTA schemes, hence achieving much faster practical performance. Can we further improve this scheme? 
\end{itemize}

\subsection{Contributions}

The above questions are the main motivations of this paper, and our contributions are summarized below.

\paragraph{A sequence-convergent FISTA scheme}
By studying the $\tk$ updating rule \eqref{eq:fista-bt} of FISTA-BT and its difference with~\eqref{eq:fista-cd}, we propose a modified FISTA scheme which applies the following rule,
\beq\label{eq:pqr-sec1}
\textstyle \tk = \frac{p + \sqrt{q + r\tkm^2}}{2}  , ~~~  \ak = \frac{\tkm-1}{\tk} ,
\eeq
where $p, q \in ]0, 1]$ and $r \in ]0, 4]$, see also Algorithm \ref{alg:fista-mod}. 
Such a modification has two advantages when $r=4$,
\begin{itemize}
	\item It maintains the $O(1/k^2)$ (actually $o(1/k^2)$) convergence rate of the original FISTA-BT (Theorem~\ref{thm:rate-obj});
	\item It allows us to prove the convergence of the iterates $\seq{\xk}$ (Theorem~\ref{thm:rate-seq});
\end{itemize}
It also allows us to show that the original FISTA-BT is also optimal in terms of the constant which appears in the $O(1/k^2)$ rate, see \eqref{eq:big-O-obj} in Theorem~\ref{thm:rate-obj}.

\paragraph{Lazy-start strategy}
For the proposed scheme and FISTA-CD, owing to the free parameters in computing $\tk$, we propose in Section~\ref{sec:lazy} a so-called ``lazy-start'' strategy for practical acceleration. 
The idea of such strategy is to slow down the speed of $\ak$ approaching $1$, which can lead to a faster practical performance. For certain problems, such a strategy can be an order faster than the original schemes, see Section~\ref{sec:experiment} for illustration. 
For least squares problems, we show that theoretically there exists optimal choices for $\ak$ update which only depends on the stopping criteria. 

\paragraph{Adaptive and greedy acceleration}
Although the lazy-start strategy can significantly speed up the performance of FISTA, it still suffers the oscillatory behavior since the inertial parameter $\ak$ eventually converges to $1$. 
By combining with the restarting technique of \cite{o2012adaptive}, in Section~\ref{sec:adaptive} we propose two different acceleration strategies: restarting adaptation to (local) strong convexity and greedy scheme.

The oscillatory behavior of FISTA schemes is often related to strong convexity. When the problem is strongly convex, there exists an optimal choice $\asol < 1$ for $\ak$ \cite{nesterov2004introductory}, 
Moreover, under such $\asol$ the iteration will no longer oscillate. 
Many problems in practice are only locally strongly convex however, estimating strong convexity in general is time consuming. 
Therefore in Section~\ref{sec:adaptive}, we propose an adaptive scheme (Algorithm~\ref{alg:rada-fista}) which combines the restarting technique \cite{o2012adaptive} and parameter update rule \eqref{eq:pqr-sec1}. Such an adaptive scheme avoids the direct estimation of strong convexity and achieve state-of-the-art performance.

We also investigate the mechanism of oscillation and the restarting technique, and propose a greedy scheme (see Algorithm~\ref{alg:greedy}) which uses aggressive inertial parameter (\eg $\ak \geq 1$) and step-size (\eg $\gamma \geq 1/L$), hence probing the limit of the restarting technique. Doing so, the greedy scheme can achieve a faster practical performance than the restarting FISTA of \cite{o2012adaptive}.

\paragraph{Nesterov's accelerated schemes}
Given the close relation between FISTA and the Nesterov's accelerated schemes~\cite{nesterov2004introductory}, we also extend the above results, particularly the modified FISTA to Nesterov's schemes. 
Such an extension can also significantly improve the performance when compared to the original schemes.

\subsection{Paper organization}
The rest of the paper is organized as follows. 
Some notation and preliminary results are collected in Section~\ref{sec:preliminary}. 
The proposed sequence-convergent FISTA scheme is presented in Section~\ref{sec:fista-mod}. The lazy-start strategy and the adaptive/greedy acceleration schemes are presented in Section~\ref{sec:lazy} and Section~\ref{sec:adaptive} respectively. In Section~\ref{sec:nesterov}, we extend the results to Nesterov's accelerated schemes. Numerical experiments are presented in Section~\ref{sec:experiment}.

%% file: tex/sec-preliminary.tex

\section{Preliminaries}\label{sec:preliminary}

Throughout the paper, $\calH$ is a real Hilbert space equipped with scalar product $\iprod{\cdot}{\cdot}$ and norm $\norm{\cdot}$. $\Id$ denotes the identity operator on $\calH$.
%
%
$\bbN$ is the set of non-negative integers and $k \in \bbN$ is the index, $\xsol \in \Argmin(\Phi)$ denotes a global minimizer of \eqref{eq:min-problem}. 


The sub-differential of a proper convex and lower semi-continuous function $R: \calH \to ]-\infty, +\infty]$ is a set-valued mapping defined by
\beq\label{eq:partial-R}
\partial R :  \calH \setvalued \calH ,~ x\mapsto \Ba{ g\in\calH ~|~ R(x') \geq R(x) + \iprod{g}{x'-x} ,~~ \forall x' \in \calH }.
\eeq

\begin{definition}[Monotone operator]\label{def:mon-opt}
	A set-valued mapping $A : \calH \setvalued \calH$ is said to be monotone if, 
	\beq\label{eq:monotone}
	\iprod{x_1-x_2}{v_1-v_2} \geq 0  ,\quad \forall~ v_1 \in A(x_1) ~~\textrm{and}~~ v_2 \in A(x_2)  .
	\eeq
	It is maximal monotone if the graph of $A$ can not be contained in the graph of any other monotone operators.
\end{definition}
It is well-known that for proper, convex and lower semi-continuous function $R: \calH \to ]-\infty, +\infty]$, its sub-differential is maximal monotone \cite{rockafellar1997convex}, and that $\prox_{R}=(\Id + \partial R)^{-1}$.

\begin{definition}[Cocoercive operator]\label{def:coco-opt}
	Let $\beta\in ]0,\pinf[$ and $B:\calH\rarrow\calH$, then $B$ is $\beta$-cocoercive if 
	\beq\label{eq:cocoercive}
	\iprod{B(x_1)-B(x_2)}{x_1-x_2} \geq \beta\norm{B(x_1)-B(x_2)}^2   ,~~ \forall x_1, x_2 \in\calH .
	\eeq
\end{definition}


The $L$-Lipschitz continuous gradient $\nabla F$ of a convex continuously differentiable function $F$ is $\frac{1}{L}$-cocoercive~\cite{baillon1977quelques}.


\begin{lemma}[Descent lemma \cite{bertsekas1999nonlinear}]\label{lem:descent}
	Suppose that $F: \calH \to \bbR$ is convex, continuously differentiable and $\nabla F$ is $L$-Lipschitz continuous. Then, given any $x,y\in\calH$,
	\[
	F(x) \leq F(y) + \iprod{\nabla F(y)}{x-y} + \sfrac{L}{2}\norm{x-y}^2  .
	\]
\end{lemma}

Given any $x,y\in\calH$, define the energy function $E_{\gamma}(x, y)$ by
\[
\begin{aligned}
E_{\gamma}(x, y)
&\eqdef R(x) + F(y) + \iprod{x-y}{\nabla F(y)} + \sfrac{1}{2\gamma}\norm{x-y}^2  .
\end{aligned}
\]
It is obvious that $E_{\gamma}(x, y)$ is strongly convex with respect to $x$, hence denote the unique minimizer as 
\beq\label{eq:e_gamma}
\begin{aligned}
e_{\gamma}(y) 
\eqdef \argmin  \Ba{E_{\gamma}(x, y): x \in \bbR^n}  
&= \argmin_{x}\Ba{ \gamma R(x) + \sfrac{1}{2}\norm{x - \pa{y - \gamma \nabla F(y)}}^2 } \\
&= \prox_{\gamma R} \Pa{ y - \gamma \nabla F(y) } 	.
\end{aligned}
\eeq
The optimality condition of $e_{\gamma}(y) $ is described below. 

\begin{lemma}[Optimality condition of $e_{\gamma}(y)$]\label{lem:opt-y}
	Given $y \in \calH$, let $y^+ = e_{\gamma}(y)$, then 
	\[
	0 
	\in \gamma \partial R(y^+) + \Pa{y^+ - \pa{y - \gamma \nabla F(y)}}
	= \gamma \partial R(y^+) + (y^+-y) + \gamma \nabla F(y)  .
	\]
\end{lemma}

We have the following basic lemmas from \cite{fista2009}.

\begin{lemma}[{\cite[Lemma~2.3]{fista2009}}]\label{lem:ineq-y}
	Let $y \in \calH$ and $\gamma \in ]0, 2/L[$ such that
	\[
	\Phi(e_{\gamma}(y)) \leq E_{\gamma}(e_{\gamma}(y), y)  ,
	\]
	then for any $x \in \calH$, we have $\Phi(x) - \Phi(e_{\gamma}(y))
	\geq \frac{1}{2\gamma}\norm{e_{\gamma}(y) - y}^2 + \frac{1}{\gamma}\iprod{y-x}{e_{\gamma}(y) - y} $. 
\end{lemma}

\begin{lemma}[{\cite[Lemma~3.1]{chambolle2015convergence}}]\label{lem:energy-decreasing}
	Given $y \in \calH$ and $\gamma \in ]0, 1/L]$, let $y^+ = e_{\gamma}(y)$, then for any $x \in \calH$, we have
	\[
	\textstyle  \Phi(y^+) + \sfrac{1}{2\gamma} \norm{y^+ - x}^2
	\leq \Phi(y) + \sfrac{1}{2\gamma} \norm{y - x}^2  .
	\]
\end{lemma}

%% file: tex/sec-fista-mod.tex
\section{A sequence-convergent FISTA scheme}\label{sec:fista-mod}

As we mentioned in the introduction, the main problems of the current FISTA schemes are caused by the behavior of $\ak$, that $\ak$ converges to $1$ too fast. As a result, we need some proper way to control this speed. For FISTA-CD, this can be achieved by choosing a relatively large value of $d$, while for FISTA-BT there is no option so far. 
In this section, we shall first discuss how to introduce control parameters to FISTA-BT which leads to a modified FISTA scheme, and then present convergence analysis.

\subsection{A modified FISTA}\label{sec:observations}

Recall the $\tk$ update rule of the original FISTA-BT \cite{fista2009}, 
\[
\textstyle    \tk = \frac{1+\sqrt{1 + 4 t^2_{k-1} }}{2} ,~~ \ak = \frac{\tkm - 1}{\tk}   .    
\]
In the following, we replace the constants $1,1$ and $4$ in the update of $\tk$ with three parameters $p,q$ and $r$ and study how they affect the behavior of $\tk$ and consequently $\ak$.

\paragraph{Observation I}
Consider first replacing $4$ with a non-negative $r$, we get
\beq\label{eq:fista-r}
\textstyle \tk = \frac{1+\sqrt{1 + \tcr{r} t^2_{k-1} }}{2} ,~~ \ak = \frac{\tkm - 1}{\tk}  .
\eeq
With simple calculation, we obtain: 
\beq\label{eq:fista-r-ak}
\begin{aligned}
r \in ]0, 4[ &:  \tk \to \sfrac{4}{4-r} < \pinf ,~~  \ak \to \qfrac{r}{4} < 1   ,  \\
r = 4 &: \tk \approx \sfrac{k+1}{2} \to \pinf ,~~ \ak \to 1  , \\
r \in ]4,\pinf[ &:   \tk \propto  \Ppa{\sfrac{\sqrt{r}}{2}}^k \to \pinf  ,~~  \ak \to \sfrac{2}{\sqrt{r} }  < 1 ,
\end{aligned}
\eeq
which implies that $r$ controls the limiting value of $\tk$, hence that of $\ak$. 
In Figure~\ref{fig:tk-pqr}~(a), we show graphically the behavior of $\ak$ under two choices of $r$: $r=4$ and $r=3.6$. 

\begin{figure}[!ht]
	\centering
	\subfloat[Value of $\ak$ under different $r$]{ \includegraphics[width=0.375\linewidth]{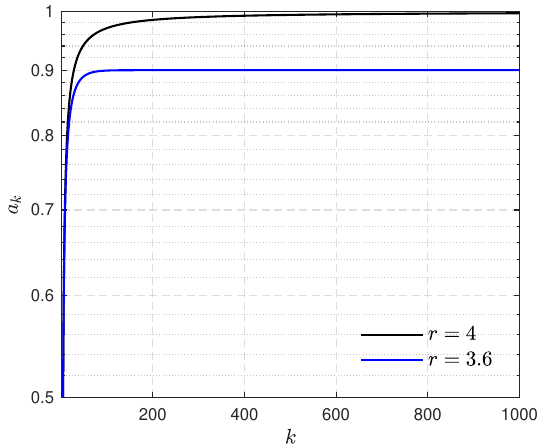} }  \hspace{24pt}
	\subfloat[Value of $\ak$ under different $p, q$]{ \includegraphics[width=0.375\linewidth]{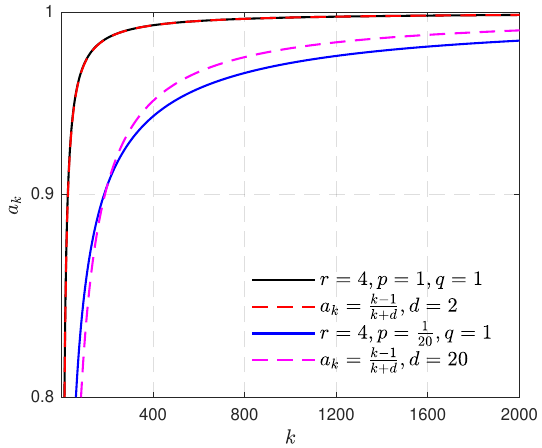} }  \\
	\caption{Different effects of $p, q$ and $r$. (a) $r$ controls the limiting value of $\ak$; (b) $p,q$ control the speed of $\ak$ approaching its limit.}
	\label{fig:tk-pqr}
\end{figure}

\paragraph{Observation II}
Now further replace the two $1$'s in \eqref{eq:fista-r} with $p, q > 0$, and restrict $r \in ]0, 4]$: 
\beq\label{eq:inertial-fista-pqr}
\textstyle \tk = \frac{p+\sqrt{q + r \tkm^2 }}{2} ,~~ \ak = \frac{\tkm - 1}{\tk}  .
\eeq
Depending on the choices of $p,q$ and $r$, this time we have 
\beq\label{eq:fista-pqr-ak}
\begin{aligned}
r \in ]0, 4[ &:  \tk \to \sfrac{2p + \Delta}{4-r} < \pinf  ,~~  \ak \to \sfrac{2p + \Delta - (4-r)}{2p + \Delta} < 1   ,  \\
r = 4 &:   \tk \approx  \sfrac{k+1}{2} p \to \pinf  ,~~  \ak \to 1  ,
\end{aligned}
\eeq
where $\Delta \eqdef \sqrt{rp^2 + (4-r)q}$.

Equation \eqref{eq:fista-pqr-ak} is quite similar to \eqref{eq:fista-r-ak}, in the sense that $\ak$ converges to $1$ for $r=4$ and to some value smaller than $1$ when $r < 4$. 
Moreover, for $r=4$, the growth of $\tk$ is controlled by $p$, indicating that we can control the speed of $\ak$ approaching $1$ via $p$, which is illustrated graphically in Figure~\ref{fig:tk-pqr}~(b). 
Under $r=4$, two different choices of $p,q$ are considered, $(p, q) = (1,1)$ and $(p,q) = (\frac{1}{20}, 1)$. Clearly, $\ak$ approaches~$1$ much slower for the second choice of $p,q$. 
In comparison, we also add a case for \eqref{eq:fista-cd} of FISTA-CD, for which a larger value of $d$ leads to a slower speed of $\ak$ approaching $1$.

\begin{remark}\label{rmk:tk-monotone}
	Let $r<4$, and denote ${t}_{\infty} \eqdef \frac{2p + \Delta}{4-r}, {a}_{\infty} = \frac{2p + \Delta - (4-r)}{2p + \Delta}$ the limiting value of $\tk, \ak$, respectively. Depending on the initial value of $t_{0}$, we have
	$
	\left\{
	\begin{aligned}
	t_0 < {t}_{\infty} &: \tk \nearrow {t}_{\infty},~ \ak \nearrow {a}_{\infty} ; \\
	t_0 = {t}_{\infty} &: \tk \equiv {t}_{\infty},~ \ak \equiv {a}_{\infty} ; \\
	t_0 > {t}_{\infty} &: \tk \searrow {t}_{\infty},~ \ak \searrow {a}_{\infty} .
	\end{aligned}
	\right.
	$
\end{remark}

\begin{center}
\begin{minipage}{0.95\linewidth}
\begin{algorithm}[H]
\caption{A modified FISTA scheme} \label{alg:fista-mod}
{\noindent{\bf{Initial}}}: $p, q > 0$ and $r \in ]0, 4]$, $t_0 = 1$, $\gamma \leq 1/L$ and $x_{0} \in \bbR^n, x_{-1} = x_{0}$. \\ 
\Repeat{convergence}{ \vspace{-1em}
\beq \label{eq:fista-mod}
\begin{gathered}
\textstyle \tk = \frac{p + \sqrt{q + r\tkm^2}}{2}  , ~~~  \ak = \frac{\tkm-1}{\tk}  ,  \\
\yk = \xk + \ak\pa{\xk-\xkm}    ,  \\
\xkp = \prox_{\gamma R}\Pa{\yk - \gamma \nabla F(\yk)} . 
\end{gathered}
\eeq
}
\end{algorithm}
\end{minipage}
\end{center}

\paragraph{A modified FISTA scheme}

Based on the above two observations of $\tk$, we propose a modified FISTA scheme, which we call ``FISTA-Mod'' for short and describe below in Algorithm~\ref{alg:fista-mod}.

\begin{remark}
When $r$ is strictly smaller than $4$, Algorithm \ref{alg:fista-mod} is simply a variant of the inertial Forward--Backward, and we refer to \cite{liang2017activity} for more details on its convergence properties. 
\end{remark}

\subsection{Convergence properties of FISTA-Mod}
\label{sec:global-convergence}

The parameters $p,q$ and $r$ in FISTA-Mod allow us to control the behavior of $\tk$ and $\ak$, hence providing possibilities to prove the convergence of the iterates $\seq{\xk}$. 
Below we provide two convergence results for Algorithm~\ref{alg:fista-mod}: $o(1/k^2)$ convergence rate for $\Phi(\xk) - \min_{x\in\calH}\Phi(x)$ and convergence of $\seq{\xk}$ together with $o(1/k)$ rate for $\norm{\xk-\xkm}$. The proofs of these results are inspired by the work of~\cite{chambolle2015convergence,AttouchFISTA15}, and for the sake of self-consistency we present the details of the proofs.

\subsubsection{Main result}

We present below first the main result, and then provide the corresponding proofs. Let $\xsol \in \Argmin(\Phi)$ be a global minimizer of the problem. 

\begin{theorem}[Convergence of objective]\label{thm:rate-obj}
	For the FISTA-Mod scheme \eqref{eq:fista-mod}, let $r = 4$ and choose $p \in ]0, 1], q > 0$ such that
	\beq\label{eq:q-ineq}
	q \leq (2-p)^2 ,	
	\eeq
	then it holds
	\beq\label{eq:big-O-obj}
	\textstyle  \Phi(\xk) - \Phi(\xsol)
	\leq   \sfrac{2L}{p^2(k+1)^2} \norm{x_{0} - \xsol}^2  .
	\eeq
	Moreover, if $p \in ]0, 1[$ and $q \in [p^2, (2-p)^2]$, then $\Phi(\xk) - \Phi(\xsol) =  o(1/k^2)$. 
\end{theorem}

\begin{remark}
	The $O(1/k^2)$ convergence rate \eqref{eq:big-O-obj} recovers the result of FISTA-BT \cite{fista2009} for $p = 1$. Since $p$ appears in the denominator, this suggests that FISTA-BT has the \emph{smallest} constant in the $O(1/k^2)$ rate. 
%
\end{remark}

\begin{theorem}[Convergence of sequence]\label{thm:rate-seq}
	For the FISTA-Mod scheme \eqref{eq:fista-mod}, let $r=4, p \in ]0, 1[$ and $q \in [p^2, (2-p)^2]$, then the sequence $\sequence{\xk}$ generated by FISTA-Mod converges weakly to a global minimizer $\xsol$ of $\Phi$. Moreover, $\norm{\xk - \xkm} =  o(1/k)$. 
\end{theorem}

\subsubsection{Proofs of Theorem \ref{thm:rate-obj}}

Before presenting the proof of Theorem \ref{thm:rate-obj}, we recall the key points for establishing $O(1/k^2)$ convergence for FISTA-BT \cite{fista2009} and $o(1/k^2)$ convergence rate \cite{chambolle2015convergence,AttouchFISTA15}. 
In particular: 
\begin{itemize}
	\item $\tk$ grows to $+\infty$ at a proper speed, \eg $\tk \approx \frac{k+1}{2}$ as pointed out in \cite{fista2009};
	\item The sequence $\seq{\tk}$ satisfies $\tk^2 - \tk \leq \tkm^2$. 
	For example, for $\tk = \frac{1+\sqrt{1+4\tkm^2}}{2}$, one has $\tk^2 - \tk = \tkm^2$.
\end{itemize}
To further improve the $O(1/k^2)$ convergence rate to $o(1/k^2)$, the key is that the difference $\tkm^2 - (\tk^2-\tk)$ should also grow to $+\infty$ \cite{chambolle2015convergence,AttouchFISTA15}. For instance, for the FISTA-CD update rule \eqref{eq:fista-cd}, one has
\beqn
\tkm^2 - (\tk^2-\tk) = \sfrac{1}{d^2}\Pa{ (d-2)k + d^2 - 3d + 3 }  ,
\eeqn
which goes to $+\infty$ as long as $d > 2$ \cite[Eq. (13)]{chambolle2015convergence}. 
It is worth noting that $\tkm^2 - (\tk^2-\tk) \to +\infty$ is also the key for proving the convergence of the iterates $\seq{\xk}$.

We start with the following supporting lemmas. Recall in \eqref{eq:fista-pqr-ak} that $\tk \approx \frac{k+1}{2} p$, we show in the lemma below that $\frac{k+1}{2} p$ is actually a lower bound of $\tk$.

\begin{lemma}[Lower bound of $\tk$]\label{lem:tk-lower-bound}
	For the $\tk$ update rule \eqref{eq:inertial-fista-pqr}, set $r = 4$ and $p \in ]0,1 ], q > 0$. Let $t_0 = 1$, then for all $k \in \bbN$, it holds that
	\beq\label{eq:tk-bnd-below}
	\tk \geq \sfrac{(k+1)p}{2}  . 
	\eeq
\end{lemma}
\begin{remark}
	When $p = 1$, we have $\tk \geq \frac{k+1}{2}$ which recovers \cite[Lemma~4.3]{fista2009}. 
\end{remark}
\begin{proof}
	Since $ p \in ]0, 1]$, it is obvious that $t_0 = 1 \geq \frac{p}{2}$ 
	and $t_1 = \frac{p + \sqrt{q+4}}{2} \geq \frac{p+2}{2} \geq p$. 
	Now suppose \eqref{eq:tk-bnd-below} holds for a given $k \in \bbN$, \ie $\tk \geq \frac{(k+1)p}{2}$. 
	Then for $k+1$, we have
	$
	\tkp - \frac{p}{2}
	= \frac{p+\sqrt{q + 4 \tk^2 }}{2} - \frac{p}{2}
	> \frac{p+2 \tk }{2} - \frac{p}{2}
	= \tk
	$ 
	which concludes the proof.
\end{proof}

\begin{lemma}[Lower bound of $\tkm^2-(\tk^2-\tk)$]\label{lem:tk-diff-bound}
	For the $\tk$ update rule \eqref{eq:inertial-fista-pqr}, let $r = 4$ and $p \in [0, 1], p^2-q \leq 0$. Then there holds
	\beq\label{eq:tk-diff-bnd-above}
	\sfrac{p(1-p)(k+1)}{2} \leq \tkm^2 - \pa{\tk^2 - \tk} .
	\eeq
\end{lemma}
\begin{remark}
	The inequality \eqref{eq:tk-diff-bnd-above} implies that, if we choose $p < 1$, then $\tkm^2 - \pa{\tk^2 - \tk} \to \pinf$.
\end{remark}
\begin{proof}
	For \eqref{eq:inertial-fista-pqr}, when $r = 4$, we have $\tk = \frac{p+\sqrt{q + 4 \tkm^2 }}{2}
	~\Leftrightarrow~ \tk^2 - p\tk + \frac{1}{4}\pa{p^2-q} = \tkm^2$. 
	Since $p^2 \leq q$, then
	\beq\label{eq:tkm-tk2-tk}
	\begin{aligned}
		\tk^2 - p \tk + \sfrac{1}{4}\pa{p^2-q} = \tkm^2 
		~~\Longrightarrow~~~~~&
		\tk^2 - p \tk \leq \tkm^2   \\
		~~\Longleftrightarrow~~~~~&
		\tk^2 - \tk + (1-p)\tk \leq \tkm^2   \\
		~~\Longrightarrow~~~~~&
		(1-p)\tk \leq \tkm^2 - \pa{\tk^2 - \tk}  \\
		{\scriptsize\textrm{(Lemma \ref{lem:tk-lower-bound})}}~\Longrightarrow~~~~~&
		\sfrac{p(1-p)(k+1)}{2}  \leq (1-p)\tk \leq \tkm^2 - \pa{\tk^2 - \tk}   ,
	\end{aligned}
	\eeq
	which concludes the proof.
\end{proof}

\begin{remark}
	The first line of \eqref{eq:tkm-tk2-tk} implies that $\tk^2 - \tkm^2 \leq p \tk$. 
	Recently it is shown in \cite{attouch2018inertial} that $p<1$ is the key for proving the convergence of the iterates $\seq{\xk}$, see \cite[Theorem 2.1]{attouch2018inertial}.
\end{remark}

The proof below is a combination of the result of \cite{fista2009,chambolle2015convergence}.

\begin{proof}[Proofs of Theorem \ref{thm:rate-obj}]

For \eqref{eq:inertial-fista-pqr}, when $r = 4$, $\tk$ is {monotonically increasing} as $\tk - \tkm \geq \frac{p}{2} >~0$. Moreover, 
\[
\begin{aligned}
\tk^2 - p \tk + \sfrac{1}{4}\pa{p^2-q} =  \tkm^2
~~\Longleftrightarrow~~&
\tk^2 - \tk + (1-p)\tk + \sfrac{1}{4}\pa{p^2-q} =  \tkm^2  \\
~~\Longrightarrow~~&
\tk^2 - \tk + (1-p)t_{0} + \sfrac{1}{4}\pa{p^2-q} \leq  \tkm^2  \\
{\scriptsize \textrm{($t_0 = 1$)}} ~\Longleftrightarrow~~&
\tk^2 - \tk + \sfrac{1}{4}\pa{ (2-p)^2 - q } \leq  \tkm^2  \\
{\scriptsize \textrm{(owing to \eqref{eq:q-ineq})}} ~\Longrightarrow~~&
\tk^2 - \tk \leq  \tkm^2  .
\end{aligned}
\]
Define $\vk = \Phi(\xk) - \Phi(\xsol)$. 
Applying Lemma~\ref{lem:ineq-y} at the points ($x = \xk, y = \yk$) and at ($x = \xsol, y = \yk$) leads to
\[
\begin{aligned}
\sfrac{2}{L} \pa{\vk - \vkp} 
&\geq \norm{\xkp-\yk}^2 + 2\iprod{\xkp-\yk}{\yk - \xk}  \\
-\sfrac{2}{L} \vkp 
&\geq \norm{\xkp-\yk}^2 + 2\iprod{\xkp-\yk}{\yk - \xsol}  ,
\end{aligned}
\]
where $\xkp = e_{\gamma}(\yk)$ is used. 
Multiplying $\tk-1$ to the first inequality and then adding to the second one yield,
\[
\sfrac{2}{L} \Pa{(\tk-1)\vk - \tk\vkp} 
\geq \tk\norm{\xkp-\yk}^2+ 2\iprod{\xkp-\yk}{\tk\yk - (\tk-1)\xk - \xsol}  .
\]
Multiply $\tk$ to both sides of the above inequality and use the result $\tk^2 - \tk \leq  \tkm^2$, we get
\[
\sfrac{2}{L} \Pa{\tkm^2\vk - \tk^2\vkp} 
\geq \tk^2\norm{\xkp-\yk}^2+ 2\tk\iprod{\xkp-\yk}{\tk\yk - (\tk-1)\xk - \xsol}  .
\]
Apply the Pythagoras relation $2\iprod{b-a}{a-c} = \norm{b-c}^2 - \norm{a-b}^2 - \norm{a-c}^2$ 
to the last inner product of the above inequality we get
\beq\label{eq:main-ineq}
\begin{aligned}
	\sfrac{2}{L} \Pa{\tkm^2\vk - \tk^2\vkp} 
	&\geq  \norm{\tk\xkp - (\tk-1)\xk - \xsol}^2 - \norm{\tk\yk - (\tk-1)\xk - \xsol}^2   \\
	&=  \norm{\tk\xkp - (\tk-1)\xk - \xsol}^2 - \norm{\tkm\xk - (\tkm-1)\xkm - \xsol}^2   .
\end{aligned}
\eeq
If $\ak - \akp \geq \bkp - \bk$ and $a_1 + b_1 < c$, then $a_k < c$ for all $k \geq 1$~\cite[Lemma~4.2]{fista2009}. 
Hence, \eqref{eq:main-ineq} yields, 
\[
\sfrac{2}{L} \tk^2\vk
\leq  \norm{x_{0} - \xsol}  .
\]
Apply Lemma \ref{lem:tk-lower-bound}, we get
\[
\Phi(\xk) - \Phi(\xsol)
\leq \sfrac{2L}{p^2(k+1)^2} \norm{x_{0} - \xsol}^2  ,
\]
which concludes the proof for the first claim \eqref{eq:big-O-obj}.


Let $\uk = \xk + \tk (\xkp-\xk)$. Applying Lemma~\ref{lem:energy-decreasing} with $y = \yk, y^+ = \xkp$ and $x = (1-\frac{1}{\tk})\xk + \frac{1}{\tk}\xsol$ yields
\[
\Phi(\xkp) + \sfrac{1}{2\gamma} \norm{\tfrac{1}{\tk}\uk - \tfrac{1}{\tk}\xsol}^2
\leq \Phi\Pa{(1-\tfrac{1}{\tk})\xk + \tfrac{1}{\tk}\xsol} + \sfrac{1}{2\gamma} \norm{\tfrac{1}{\tk}\ukm - \tfrac{1}{\tk}\xsol}^2		.
\]
Applying the convexity of $\Phi$, we further get
\[
\Pa{ \Phi(\xkp) - \Phi(\xsol) } - (1-\tfrac{1}{\tk}) \Pa{ \Phi(\xk) - \Phi(\xsol) }
\leq \sfrac{1}{2\gamma\tk^2} \Pa{ \norm{\ukm - \xsol}^2 - \norm{\uk - \xsol}^2 }		.
\]
Multiply $\tk^2$ to both sides of the above inequality, 
\[
\tk^2\Pa{ \Phi(\xkp) - \Phi(\xsol) } - (\tk^2 - \tk) \Pa{ \Phi(\xk) - \Phi(\xsol) }
\leq \sfrac{1}{2\gamma} \Pa{ \norm{\ukm - \xsol}^2 - \norm{\uk - \xsol}^2 }		.
\]
From Lemma \ref{lem:tk-diff-bound}, we have $\frac{p(1-p)(k+1)}{2} - \tkm^2 \leq  - \pa{\tk^2 - \tk}$, hence
\[
\begin{aligned}
\tk^2\Pa{ \Phi(\xkp) - \Phi(\xsol) } - \tkm^2 \Pa{ \Phi(\xk) - \Phi(\xsol) } + \sfrac{p(1-p)(k+1)}{2} \Pa{ \Phi(\xk) - \Phi(\xsol) }		
\leq \sfrac{1}{2\gamma} \Pa{ \norm{\ukm - \xsol}^2 - \norm{\uk - \xsol}^2 }		.
\end{aligned}
\]
Summing the inequality from $k=1$ to $K$, we get
\[
\begin{aligned}
&t_{K}^2\Pa{ \Phi(x_{K+1}) - \Phi(\xsol) }  + \sfrac{p(1-p)}{2} \msum_{j=1}^{K} j \Pa{ \Phi(x_{j}) - \Phi(\xsol) }		
\leq \sfrac{1}{2\gamma} \Pa{ \norm{u_{0} - \xsol}^2 - \norm{u_{K} - \xsol}^2 }		,
\end{aligned}
\]
which means that $\sum_{j=1}^{\pinf} j \Pa{ \Phi(x_{j}) - \Phi(\xsol) } < \pinf$, that is  $\Phi(\xk) - \Phi(\xsol) = o(1/k^2)$.
\end{proof}

\subsubsection{Proofs of Theorem \ref{thm:rate-seq}}

The proof of Theorem \ref{thm:rate-seq} is inspired by \cite{chambolle2015convergence}, where the authors showed that the key to prove the convergence of $\seq{\xk}$ is the following summability
\[
\msum_{k\in\bbN} k \norm{\xk-\xkm}^2 < + \infty	.
\]
As previously mentioned, the major difference between FISTA-BT \eqref{eq:fista-bt} and FISTA-CD~\eqref{eq:fista-cd} is that $\tkm^2 - (\tk^2-\tk) \to +\infty$ holds for FISTA-CD. 
For the proposed FISTA-Mod scheme, as $\frac{p(1-p)k}{2} \leq \tkm^2 - (\tk^2-\tk)$ also goes to $+\infty$ as long as $p$ is strictly smaller than $1$, this allows us to adapt the proof of~\cite{chambolle2015convergence} to FISTA-Mod, hence proving the convergence of $\seq{\xk}$.

We need two supporting lemmas before presenting the proof of Theorem~\ref{thm:rate-seq}. 
Given $\ell \in \bbN_+$, define the truncated sum $S_{\ell} \eqdef \frac{q}{4p} \sum_{i=0}^{\ell} \frac{1}{1+i}$
and a new sequence $\bartk$ by
\[
\textstyle  \bartk 
\eqdef 1+S_{\ell} + \Pa{ \sfrac{p}{2} + \sfrac{q}{4p(\ell+1)} } k  .
\]
We have the following lemma showing that $\bartk$ serves an upper bound of $\tk$.

\begin{lemma}[Upper bound of $\tk$]\label{lem:tk-upper-bound}
	For the $\tk$ update rule \eqref{eq:inertial-fista-pqr}, let $r = 4$ and $p, q \in [0,1 ]$. For all $k \in \bbN$, it holds that $\tk \leq \bartk$. 
\end{lemma}
The purpose of bounding $\tk$ from above by a linear function of $k$ is such that we can eventually bound $\ak$ from above, which is needed by the following lemma. 
\begin{proof}
	Given $\tk, \tkp$, we have 
	\[
	\textstyle \tkp - \tk
	= \frac{p + \sqrt{q+4\tk^2}}{2} - \tk 
	= \frac{p}{2} +  \frac{\sqrt{q+4\tk^2} - 2\tk}{2}  
	\leq \qfrac{p}{2} +\frac{\sqrt{ (2\tk+q/(4\tk))^2 } - 2\tk}{2} 
	= \frac{p}{2} + \frac{q}{8\tk}	. 
	\]
	Clearly, $t_0 \leq \bart_{0}$. Suppose $\tk \leq \bartk$ for $\ell \leq k$ and recall that $\tk \geq \frac{k+1}{2} p$, then we have
	\[
	\begin{aligned}
	\tkp 
	\leq \tk + \sfrac{p}{2} + \sfrac{q}{8\tk}
	\leq \bartk + \sfrac{p}{2} + \sfrac{q}{8\tk}
	&= 1+S_{\ell} + \Pa{ \sfrac{p}{2} + \sfrac{q}{4p(\ell+1)} } k + \sfrac{p}{2} + \sfrac{q}{8\tk}	\\
	&\leq 1+S_{\ell} + \Pa{ \sfrac{p}{2} + \sfrac{q}{4p(\ell+1)} } k + \sfrac{p}{2} + \sfrac{q}{4{(k+1)} p}	\\
	&\leq 1+S_{\ell} + \Pa{ \sfrac{p}{2} + \sfrac{q}{4p(\ell+1)} } k + \sfrac{p}{2} + \sfrac{q}{4{(\ell+1)} p} 	
	= \bart_{k+1}	,  
	\end{aligned}
	\]
	and we conclude the proof.
\end{proof}

Denote $\ca{x}$ the smallest integer that is larger than $x$, and define the following two constants
\[
\textstyle b \eqdef \ca{ \sfrac{p+2}{p + {q}/\pa{2p(\ell+1)}} } 
\qandq
c \eqdef \sfrac{p + 2 + 2S_{\ell}}{p + {q}/\pa{2p(\ell+1)}}  .
\]

\begin{lemma}\label{lem:sum-betajk}
	For all $j \geq 1$, define $\beta_{j,k} \eqdef \mprod_{i=j}^{k} a_{i} $ 
	for all $j, k$, and $\beta_{j, k} = 1$ for all $k < j$. 
	Let $\ell \geq \ca{\frac{q}{p(2-p)}}$, then for all $j$, it holds that $\msum_{k=j}^{\infty} \beta_{j, k} 
	\leq j + c + 2b$. 
\end{lemma}
\begin{proof}
	We first show that $\ak$ is bounded from above. From the definition of $\ak$ we have
	\beq\label{eq:ak-up-bnd}
	\begin{aligned}
		\ak
		= \sfrac{\tkm - 1}{ \tk } 
		= \sfrac{2\tkm - 2}{ p + \ssqrt{q + 4\tkm^2} }   
		\leq \sfrac{p+2\tkm - 2-p}{ p + 2\tkm }   
		&= 1 - \sfrac{2+p}{ p + 2\tkm }   \\
		{\hspace{4.3em}} { \textrm{\scriptsize (Lemma~\ref{lem:tk-upper-bound})}} 
		& \leq 1 - \sfrac{2+p}{ p + 2 + 2S_{\ell} + \pa{ p + \frac{q}{2p(\ell+1)} } k }   
		= 1 - \sfrac{b}{ k + c }    .
	\end{aligned}
	\eeq
	From \eqref{eq:ak-up-bnd} we have that
	\[
	\beta_{j,k} 
	= \mprod_{i=j}^{k} a_{i}   
	\leq \mprod_{i=j}^{k}  \sfrac{i+c-b}{ i + c }  .
	\]
	For $k=j,...,j+2b-1$, we have $\beta_{j,k} < 1$. Then for $k-j \geq 2b$,
	\[
	\begin{aligned}
	\beta_{j,k} 
	\leq \mprod_{i=j}^{k}  \sfrac{i + c - b}{ i + c }  
	%
	&=  \sfrac{j + c - b}{ j + c } ~
	\sfrac{j + 1 + c - b}{ j + 1 + c }
	\dotsm
	\sfrac{j + c }{ j + b + c }	~
	\sfrac{j + 1 + c}{ j + b+1 + c }
	\dotsm
	\sfrac{k + c - b}{ k + c } \\
	&=  \sfrac{(j + c - b) \dotsm (j + c - 1)}{ (k+c - b+1) \dotsm (k+c) }  
	\leq \sfrac{(j + c -1 )^{b}}{ (k + c - b + 1)^{b} }   .
	\end{aligned}
	\]
	Therefore, 
	\[
	\begin{aligned}
	\msum_{k=j}^{\infty} \beta_{j, k}
	\leq 2b +  \msum_{k=j+2b}^{\infty} \beta_{j, k}
	&\leq 2b +  (j + c - 1 )^{b} \msum_{k=j+2b}^{\infty} \sfrac{1}{ (k + c -b+1)^{b} }   \\
	&\leq 2b +  (j + c - 1 )^{b} \int_{x=j+2b}^{\infty} \sfrac{1}{ (x+ c-b+1)^{b} } \mathrm{d}x   \\
	&\leq 2b +  (j + c - 1 )^{b} \sfrac{1}{b-1} \sfrac{1}{ (j + b+ c +1)^{b-1} }   \\
	&\leq 2b +  \sfrac{1}{b-1} (j + c - 1 )  
	\leq j + c + 2b  .
	\end{aligned}
	\]
	The last inequality uses the fact that $b \geq 2$ for $\ell \geq \ca{\frac{q}{p(2-p)}}$. 
\end{proof}

\begin{proof}[Proofs of Theorem \ref{thm:rate-seq}]
	
	Applying Lemma~\ref{lem:energy-decreasing} with $y = \yk$ and $x = \xk$, we get
	\[
	\textstyle \Phi(\xkp) + \sfrac{\norm{\xk-\xkp}^2}{2\gamma}
	\leq \Phi(\xk) + \ak^2\sfrac{\norm{\xkm-\xk}^2}{2\gamma}  ,
	\]
	which means, let $\Delta_{k} \eqdef \frac{1}{2} \norm{\xk-\xkm}^2$, that $\Delta_{k+1} - \ak^2\Delta_{k}
	\leq \gamma\pa{\vk - \vkp}$. 
	%
	Denote the upper bound of $\ak$ in \eqref{eq:ak-up-bnd} as $\abark \eqdef 1 - \frac{b}{ k + c } ,~ \forall k \geq 2$, 
	and let $\abar_{1} = 0$ since $a_{1} = 0$. 
	It is then straightforward that
	\[
	\Delta_{k+1} - \abark^2\Delta_{k}
	\leq \Delta_{k+1} - \ak^2\Delta_{k}
	\leq \gamma\pa{\vk - \vkp}  .
	\]
	Multiplying the above inequality with $(k+c)^2$ and summing from $k=1$ to $K$ lead to
	\[
	\msum_{k=1}^{K}(k+c)^2\pa{\Delta_{k+1} - \abark^2\Delta_{k}}
	\leq \gamma \msum_{k=1}^{K} (k+c)^2 \pa{\vk - \vkp}  .
	\]
	Since $\abar_{1} = 0$, we derive from above that
	\[
	\begin{aligned}
	\msum_{k=1}^{K}(k+c)^2\pa{\Delta_{k+1} - \abark^2\Delta_{k}} 	
	&= (K+c)^2\Delta_{K+1} + \msum_{k=2}^{K}\Pa{(k+c-1)^2 - (k+c)^2\abark^2}\Delta_{k}  \\
	&= (K+c)^2\Delta_{K+1} + \msum_{k=2}^{K}\Pa{(k+c-1)^2 - (k + c - b)^2}\Delta_{k}  \\
	&\leq (K+c)^2\Delta_{K+1} + \msum_{k=2}^{K}2(b-1)(k+c)\Delta_{k}  \\
	&\leq \gamma \Pa{ (c+1)^2 w_{1} - (c+K)^2w_{K+1} } +  \gamma \msum_{k=2}^{K} \Pa{ (k+c)^2 - (k+c-1)^2 } \vk   \\
	&\leq \gamma \Pa{ (c+1)^2 w_{1} - (c+K)^2w_{K+1} } +  2\gamma \msum_{k=2}^{K} (k+c) \vk   .
	\end{aligned}
	\]
	From the proof of Theorem~\ref{thm:rate-obj}, we have that $\sum_{k\in\bbN} k \vk < \pinf$, which in turn implies that $\sequence{k\Delta_{k}}$ is \emph{summable} and that sequence $\seq{k^2\Delta_{k}}$ is bounded, which also indicates $\norm{\xk-\xkm} = o(1/k)$.
	
Now define
	$
	\psi_{k} \eqdef \sfrac{1}{2}\norm{\xk - \xsol}^2$ and $
	\phi_{k} \eqdef \sfrac{1}{2}\norm{\yk - \xkp}^2  
	$. 
	By applying the definition of $\yk$, we have
	\beq\label{eq:phi-xk-xkp}
	\begin{aligned}
		\psi_{k} - \psi_{k+1} 
		&= \sfrac12\iprod{\xk-\xsol+\xkp-\xsol}{\xk-\xkp}  \\
		&= \Delta_{k+1} + \iprod{\yak-\xkp}{\xkp-\xsol} - \ak\iprod{\xk-\xkm}{\xkp-\xsol} \\
		&\geq \Delta_{k+1} + \gamma\iprod{\nabla F(\yk) - \nabla F(\xsol)}{\xkp-\xsol}  - \ak\iprod{\xk-\xkm}{\xkp-\xsol}.
	\end{aligned}
	\eeq
	As $\nabla F$ is $\frac{1}{L}$-cocoercive (Definition~\ref{def:coco-opt}), applying Young's inequality yields
	\beq \label{eq:gradF-cocoercive}
	\begin{aligned}
		\iprod{\nabla F(\yk)-\nabla F(\xsol)}{\xkp-\xsol}  
		&\geq {\sfrac{1}{L}}\norm{\nabla F(\yk)-\nabla F(\xsol)}^2 + \iprod{\nabla F(\yk)-\nabla F(\xsol)}{\xkp-\yk}  \\
		&\geq {\sfrac{1}{L}}\norm{\nabla F(\yk)-\nabla F(\xsol)}^2 - {\sfrac{1}{L}}\norm{\nabla F(\yk)-\nabla F(\xsol)}^2 - \sfrac{L}{2}\phi_{k} 
		= -\sfrac{L}{2}\phi_{k}. 
	\end{aligned} 
	\eeq
	Back to \eqref{eq:phi-xk-xkp}, we get
	\beq\label{eq:phi-xk-xkp-1}
	\begin{aligned}
		\psi_{k} - \psi_{k+1} 
		&\geq \Delta_{k+1}  -  \sfrac{\gamma L}{2}\phi_{k}  - \ak\iprod{\xk-\xkm}{\xkp-\xsol}.
	\end{aligned}
	\eeq
	For $\iprod{\xk-\xkm}{\xkp-\xsol}$, we have
	\beq\label{eq:xk-xkm-xkp-xsol}
	\begin{aligned}
		\iprod{\xk-\xkm}{\xkp-\xsol} 
		&= \iprod{\xk-\xkm}{\xkp-\xk}  + \iprod{\xk-\xkm}{\xk-\xsol}  \\
		&= \iprod{\xk-\xkm}{\xkp-\xk} +\pa{\Delta_{k} + \psi_{k} - \psi_{k-1} } ,
	\end{aligned}
	\eeq
	where we applied the usual Pythagoras relation to $\iprod{\xk-\xkm}{\xk-\xsol}$. 
	Putting \eqref{eq:xk-xkm-xkp-xsol} back into \eqref{eq:phi-xk-xkp-1} and rearranging terms yield
	\beq\label{eq:phi-xk-xkp-3}  
	\begin{aligned}
		\psi_{k+1} - \psi_{k} -  \ak\pa{\psi_{k} - \psi_{k-1} }  
		&\leq -\Delta_{k+1} + \sfrac{\gamma L}{2}\phi_{k} + \ak\iprod{\xk-\xkm}{\xkp-\xk} + \ak \Delta_{k}  \\
		&= -\Delta_{k+1} + \sfrac{\gamma L}{2}\phi_{k} + \iprod{\yk-\xk}{\xkp-\xk} + \ak \Delta_{k}   \\
		&= -\Delta_{k+1} + \sfrac{\gamma L}{2}\phi_{k} + \Pa{ \ak^2\Delta_{k} + \Delta_{k+1} - \sfrac{1}{2}\norm{\yk-\xkp}^2 } + \ak \Delta_{k}   \\
		&= \sfrac{\gamma L - 1}{2}\phi_{k}  + (\ak+\ak^2) \Delta_{k} , 
	\end{aligned}
	\eeq
	where the Pythagoras relation is applied again to $\iprod{\yk-\xk}{\xkp-\xk}$. 
	Since $\gamma \in ]0, 1/L]$ and $\ak \leq 1$, we get from above that
	\[
	\psi_{k+1} - \psi_{k} -  \ak\pa{\psi_{k} - \psi_{k-1} }  
	\leq 2\ak \Delta_{k}  .
	\]
	Define $\xi_{k} = \max\ba{0, \psi_{k}-\psi_{k-1}}$, then
	\[
	\xi_{k+1} 
	\leq \ak \pa{\xi_{k} + 2\Delta_{k}} 
	\leq 2 \msum_{j=2}^{k} \bPa{ \mprod_{l=j}^{k} a_{l} } \Delta_{j}
	= 2 \msum_{j=2}^{k} \beta_{j,k} \Delta_{j}  ,
	\]
	Applying Lemma~\ref{lem:sum-betajk} and the summability of $\seq{k\Delta_{k}}$ leads to
	\[
	\msum_{k=2}^{\pinf} \xi_{k} 
	\leq 2 \msum_{k=1}^{\pinf}  \msum_{j=2}^{k} \beta_{j, k} \Delta_{j}  
	= 2 \msum_{j=2}^{k}  \Delta_{j}   \msum_{k=1}^{\pinf}   \beta_{j, k} 
	\leq 2 \msum_{j=2}^{k} \pa{j+c+2b} \Delta_{j}  
	< \pinf   .
	\]
	Then we have
	\[
	\Phi_{k+1} - \msum_{j=1}^{k+1}[\theta_{j}]_{+} \leq \Phi_{k+1} - \theta_{k+1} - \msum_{j=1}^{k}[\theta_{j}]_{+} 
	= \Phi_{k} - \msum_{j=1}^{k}[\theta_{j}]_{+} ,
	\]
	which means $\sequence{ \Phi_{k} - \sum_{j=1}^{k}[\theta_{j}]_{+}  }$ is monotone non-increasing, hence convergent. It is immediate that $\sequence{ \Phi_{k}}$ is also convergent, meaning that $\lim_{k \to +\infty}\norm{\xk - \xsol}$ exists for any $\xsol$ such that $0 \in A(\xsol)+B(\xsol)$.

	Let $\bar{x}$ be a weak cluster point of $\seq{\xk}$, and let us fix a subsequence, say $x_{k_j}\rightharpoonup \bar{x}$. 
	Applying Lemma~\ref{lem:opt-y} with $y = y_{k_j}$, we get
	\beqn
	\textstyle g_{k_j} \eqdef \frac{y_{k_j} - x_{k_j+1}}{\gamma} - \nabla F(y_{k_j}) \in \partial R(x_{k_j+1}) .
	\eeqn
	Since $\nabla F$ is cocoercive and $y_{k_j} = x_{k_j} + a_{k_j}(x_{k_j} - x_{k_j-1}) \rightharpoonup \bar{x}$, we have $\nabla F(y_{k_j}) \to \nabla F(\bar{x})$. In turn, $u_{k_j} \to -\nabla F(\bar{x})$ since $\gamma > 0$. 
	Since $(x_{k_j+1},u_{k_j}) \in \gra (\partial R)$, and the graph of the maximal monotone operator $\partial R$ is sequentially weakly-strongly closed in $\calH \times \calH$, we get that $-\nabla F(\bar{x}) \in \partial R(\bar{x})$, \ie $\bar{x}$ is a solution of \eqref{eq:min-problem}. Opial's Theorem~\cite{opial1967weak} then concludes the proof. 
\end{proof}

%% file: tex/sec-lazystart.tex

\section{Lazy-start strategy}\label{sec:lazy}

From the last section, the benefits of free parameters $p,q, r$ in FISTA-Mod are $o(1/k^2)$ convergence rate in objective function value and convergence of sequence. 
In this section, we further show that the degree of freedom provided by these parameters allows us to design a so-called ``lazy-start strategy'' which can make FISTA-Mod/FISTA-CD much faster in practice.

\begin{proposition}[Lazy-start FISTA]\label{prop:lazy-start}
	For FISTA-Mod and FISTA-CD, consider the following choices of $p,q$ and $d$ respectively:
	\begin{description}[leftmargin=3.25cm]
		\item[FISTA-Mod] $p \in [\frac{1}{80}, \frac{1}{10}], q \in [0, 1]$ and $r = 4$;
		\item[FISTA-CD] $d \in [10, 80]$.
	\end{description}
\end{proposition}

\begin{remark}
The intervals for $p$ and $d$ are obtained from practical observations and not inclusive. Take FISTA-CD for example, there can be problems where $d<10$ or $d>80$ provides even faster performances. 
\end{remark}

The main reason of calling the above strategy ``lazy-start'' is that it slows down the speed of $\ak$ converging to $1$; Recall Figure~\ref{fig:tk-pqr}~(b). 
To discuss the advantage of lazy-start, we consider the simple least square problem: 
%
\beq\label{eq:lse}
\min_{x \in \bbR^{201}} \Ba{ F(x) \eqdef \sfrac{1}{2} \norm{Ax}^2 }  ,
\eeq
where $A \in \bbR^{201\times 201}$ is of the form
\[
A = 
\begin{bmatrix}
2 & -1 & & &    \\
-1 & 2 & -1 & &   \\
& \ddots & \ddots & \ddots & \\
& & -1 & 2 & -1  \\
& & & -1 & 2  \\
\end{bmatrix}_{201\times201}   .
\]
%
In this example, $F$ is strongly convex and admits a unique minimizer $\xsol = 0$. 

In what follows, we first discuss the advantage of lazy-start in the discrete setting, and then in the continuous dynamical system setting.

\subsection{Advantage of lazy-start}

Specialising FISTA-CD to solve \eqref{eq:lse}, we get
\beq\label{eq:fista-cd-lse}
\begin{aligned}
	\yk &= \xk + \tfrac{k-1}{k + d}(\xk - \xkm)	\\
	\xkp &= \yk - \tfrac{1}{L} A^TA\yk = (\Id - \tfrac{1}{L} A^TA) \yk	 .
\end{aligned}
\eeq
To show the benefits of lazy-start, two different values of $d$ are considered:
\begin{itemize}
	\item FISTA-CD with $d = 2$;
	\item Lazy-start FISTA-CD with $d = 20$. 
\end{itemize}
%
%
The convergence of $\norm{\xk-\xsol}$ for the two choices of $d$ are plotted in Figure~\ref{fig:cmp-lse}, where the {red line} represents $d = 2$ and the {black line} for $d = 20$. 
The starting points $x_0$ for both cases are the same and chosen such that $\norm{x_0-\xsol} = 1$. 
It can be observed that the lazy-start one is significantly faster than the normal choice after iteration step $k = 2\times 10^5$.

\begin{figure}[!ht]
	\centering \includegraphics[width=0.375\linewidth]{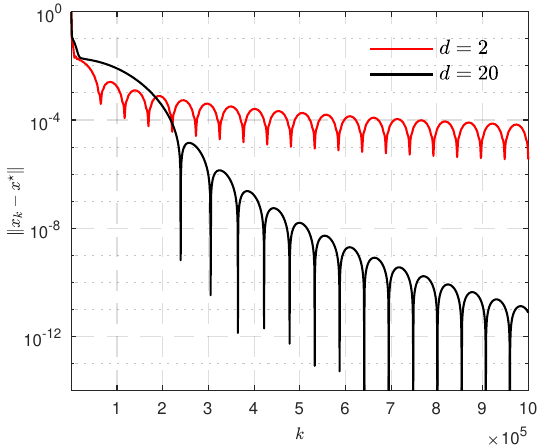}   \\
	\caption{Convergence comparison of $\norm{\xk-\xsol}$ of FISTA-CD for $d=2$ and $d=20$. } 
	\label{fig:cmp-lse}
\end{figure}

{\noindent}To explain such a difference, we need the following steps:
\begin{enumerate}[label = {\rm (\arabic{*})}]
	\item
	Fixed-point characterization of \eqref{eq:fista-cd-lse}: the iteration can be written as a linear system owing to the quadratic form of the problem; See \eqref{eq:Mk}.
	\item
	Spectral property of the linear system: the spectral property of the linear system is controlled only by $d$. 
	
	\item Advantage of lazy-start: comparison of spectral properties under different choices of $d$. 

\end{enumerate}
It is worth noting that, the convergence seen in Figure \ref{fig:cmp-lse} appears not only for \eqref{eq:lse}, but rather is observed in many problems; see Section~\ref{sec:experiment} for more examples.

\paragraph{Fixed-point formulation of \eqref{eq:fista-cd-lse}}
Denote $G = \Id - \tfrac{1}{L} A^TA$, we have from \eqref{eq:fista-cd-lse} that,
\[
\xkp - \xsol 
= G (\yk - \xsol) 
= (1 + \ak) G (\xk - \xsol) - \ak G (\xkm - \xsol)	.
\]
Define
\beq\label{eq:Mk}
\zk \eqdef \begin{pmatrix}
	\xk - \xsol \\ \xkm - \xsol	
\end{pmatrix}
\qandq
M_{d,k} \eqdef \begin{bmatrix}
	(1 + \ak) G & - \ak G		\\
	\Id & 0	
\end{bmatrix}	.
\eeq 
Then it is immediate that
\beq\label{eq:Mk-zk}
\zkp = M_{d,k} \zk	,
\eeq
which is the fixed-point characterization of \eqref{eq:fista-cd-lse}. 
Denote $\wMdk \eqdef \prod_{i=1}^{k-1} M_{d,k-i}$, then recursively apply the above relation, we get
\[
\zk = \wMdk z_{1}  .
\]

\paragraph{Spectral property of $\wMdk$}
From above it is immediate that 
\[
\norm{\zk} = \norm{\wMdk} \norm{z_{1}}  .
\]
To set up the comparison between $d=2$ and $d=20$, we need to compute spectral property of $\norm{\wMk}$: 
\begin{itemize}
\item Let $\rho_{d,i}$ be the \emph{leading eigenvalue} of $M_{d,i}$ for $i=1,...,k-1$, then there exists $\calC > 0$ such that 
	\beq\label{eq:norm-wMk}
	\norm{\wMdk} \leq \calE_{d,k} \eqdef \calC \mprod_{i=1}^{k-1} \abs{ \rho_{d,k-i} }	
	\eeq
	holds for all $k \geq 1$. We call $\calE_{d,k}$ the envelope of $\norm{\wMdk}$. Unfortunately, unlike the case of $M_{d,k}$, this time we can only discuss through numerical illustration. 
 

\item Let $\alpha$ be the smallest eigenvalue of $A^TA$ and $\eta = 1-\alpha/L$ the leading eigenvalue of $G$. Owing to the result of \cite{liang2017activity}, for each $M_{d,k}$, the magnitude of its leading eigenvalue $\rho_{d,k}$ reads: 
\beq\label{eq:mag-rho}
	\abs{\rho_{d,k}} = 
	\left\{
	\begin{aligned}
		\sfrac{ (1 + \ak)\eta + \sqrt{(1 + \ak)^2\eta^2 - 4\ak\eta} }{2} < 1 &: \ak \leq \asol	,	\\
		\sqrt{\ak \eta} < 1 &: \ak \geq \asol,
	\end{aligned}
	\right.
	\eeq
	where $\asol = \frac{1 - \sqrt{\alpha/L}}{1 + \sqrt{\alpha/L}}$. 
	Moreover, $\abs{\rho_k}$ attains the minimal value $\rho^\star = 1 - \sqrt{\alpha/L}$ when $\ak = \asol$ \cite{liang2017activity}. 
	

\end{itemize}
For more details about the dependence of $\rho_{d,k}$ on $\eta$ and $\ak$, we refer to \cite{liang2017activity,liang2016thesis}. 
Below we inspect the value of $\abs{\rho_{d,k}}$ under $d=2$ and $d=20$. 
The modulus of $\abs{\rho_{d,k}}$ for $d=2,20$ are shown in Figure~\ref{fig:why-1}~(a), where the red line is $\abs{\rho_{2,k}}$ and the black line stands for~$\abs{\rho_{20, k}}$: 
\begin{itemize}
	\item In both cases, the values of $\abs{\rho_{2,k}}, \abs{\rho_{20, k}}$ decrease first, until reaching $\rho^\star = 1 - \sqrt{\alpha/L}$, and then start to increase until they reach $\sqrt{\eta}$; 
	\item Choosing $d=20$ slows the speed at which $a_k$ is increasing (see Figure~\ref{fig:tk-pqr}), therefore also slows the speed at which $\abs{\rho_{20,k}}$ approaches $\rho^\star$. Such a difference in approach to $\rho^\star$ is key for the lazy-start strategy being faster. 
\end{itemize}
Denote $\Keq$ the point $\abs{\rho_{20, k}}$ equals to $\rho^\star$, then we have $\Keq = \ca{ \frac{1 + 20 \asol}{1 - \asol } }$. 

\begin{figure}[!ht]
	\centering
	\subfloat[Value of $\abs{\rho_{d,k}}$]{ \includegraphics[width=0.39\linewidth]{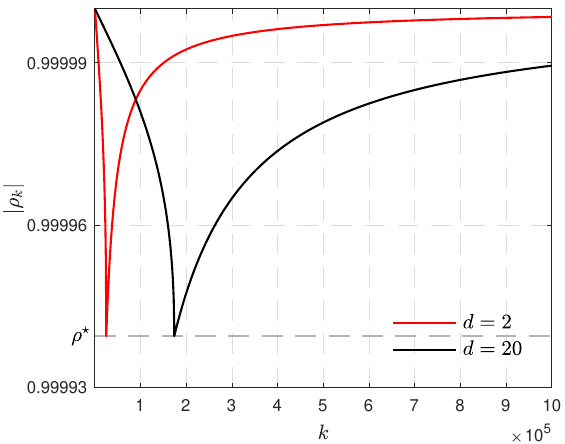} }  \hspace{24pt}
	\subfloat[Value of $\calE_{d, k}$]{ \includegraphics[width=0.375\linewidth]{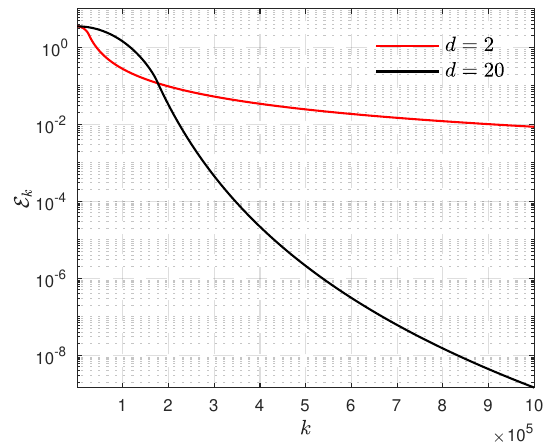} }  \\
	\caption{The value of $\abs{\rho_{d,k}}$ and $\calE_{d, k}$ under $d = 2, 20$.}
	\label{fig:why-1}
\end{figure}

\paragraph{The advantage of lazy-start}


Now we compare $\calE_{2, k}, \calE_{20, k}$, whose values are plotted in Figure~\ref{fig:why-1}~(b), where the red and black lines are corresponding to $\calE_{2, k}$ and $\calE_{20, k}$ respectively. 
Observe that, $\calE_{2, k}$ and $\calE_{20, k}$ intersect for certain $k$ which turns out very close to $\Keq$. 
For $k \geq \Keq$, the difference between $\calE_{2, k}$ and $\calE_{20, k}$ becomes increasingly large.

From \eqref{eq:mag-rho} and the definition of $\ak$, we have that for $k \geq \Keq$, 
\[
\textstyle \abs{\rho_{2, k}} 
= \sqrt{\frac{k-1}{k + 2} \eta}  
\geq
\abs{\rho_{20, k}} 
= \sqrt{\frac{k-1}{k + 20} \eta}  .
\]
Define the accumulation of $\frac{  \abs{\rho_{2, i}} }{  \abs{\rho_{20, i}} }$ by 
$
\calR_k 
\eqdef \prod_{i=\Keq}^{k} \frac{  \abs{\rho_{2, i}} }{  \abs{\rho_{20, i}} }
= \prod_{i=\Keq}^{k} \sqrt{ \frac{  i + 20 }{  i + 2 } }	
$ 
and let $k \geq \Keq + 36$, 
we get
\beq\label{eq:Rk-C}
\begin{aligned}
\calR_{k} 
&= \mprod_{i=\Keq}^{k} \sfrac{ \abs{\rho_{d_{1}, i}}  }{ \abs{\rho_{d_{2}, i}}  }
= \mprod_{i=\Keq}^{k}	\sqrt{ \sfrac{ i + 20 } { i + 2 } }		\\ 
&= \mprod_{i=\Keq}^{k} \Ppa{ \sfrac{ \tcr{\Keq + 20}} { \Keq + 2}  \sfrac{ \Keq + 1  + 20} { \Keq + 1 + 2} \dotsm  \sfrac{ \Keq + 17  + 20} { \Keq + 17  + 2} \sfrac{ \Keq + 18  + 20} { \tcr{\Keq + 18  + 2}}  \dotsm \sfrac{ k-2 + 20} { k-2 + 2} \sfrac{ k-1 + 20} { k-1 + 2} \sfrac{ k + 20} { k + 2} }^{1/2}		\\
&= \mprod_{j=0}^{17} \bPa{ \sfrac{ k + 3 + j} { \Keq + 2 + j} }^{1/2}	
\approx \BPa{ \sfrac{ k + 20} { \Keq + 19}  }^{9}	
= \Ppa{ \sfrac{ 2 }{ \sqrt{C} + 1 }  }^{9} \Ppa{ \sfrac{ k + 20 } { 21 }  }^{9} ,
\end{aligned}
\eeq
where $C \eqdef L/\alpha$ is the condition number of \eqref{eq:lse}. 
%
To verify the accuracy of the above approximation, for the considered problem \eqref{eq:lse}, we have $
L = 16		$ and $
\alpha = 5.85 \times 10^{-8}$.
Consequently, $C = \frac{L}{\alpha} = 2.735 \times 10^{8}$. Let $k = 10^{6}$ and substitute them into \eqref{eq:Rk-C}, we have $\calR_{k} \approx 5.98 \times 10^6$, while for $\calE_{d, k}$ we have 
\[
\sfrac{ \calE_{2, k=10^6} }{ \calE_{20, k=10^6} } 
= 5.96 \times 10^6 ,
\]
which means \eqref{eq:Rk-C} is a good approximation of the envelope ratio $\calE_{2, k}/\calE_{20, k}$. 

The above discussion is mainly about the envelope $\calE_{d, k}$. In terms of what really happens on $\norm{\xk-\xsol}$ for $d=2$ and $d=20$: from Figure~\ref{fig:cmp-lse}, we have that for $k=10^6$, $\norm{\xk-\xsol}$ of $d=2$ is about $2 \times 10^6$ larger than that of $d=20$. Compared with $5.98 \times 10^6$, we can conclude that \eqref{eq:Rk-C} is able to accurately estimate the order of acceleration obtained by a lazy-start strategy.

\subsection{Quantifying the advantage of lazy-start}

The approximation \eqref{eq:Rk-C} indicates that $\calR_{k}$ is a function of $C$ and $k$, in the following we discuss the dependence of $\calR_{k}$ on $C$ and $k$ from two perspectives.


\paragraph{Fixed $k$}
First consider $C \in [10^{4}, 10^{12}]$ and let $k = \Keq + 10^6$, note that $\Keq$ is changing over $C$.  
This setting is to check how much better $d=20$ is than $d=2$ in terms of $\norm{\xk-\xsol}$ if we run the iteration \eqref{eq:fista-cd-lse} $10^6$ more steps after $\Keq$. 
The obtained value of $\calR_{k}$ is shown in Figure~\ref{fig:fix-k-Keq}~(a). As we can see, when $C$ is small, \eg $C = 10^4$, the advantage can be as large as $10^{27}$ times and decrease to almost $1$ for $C = 10^{12}$. However, it should be noted that for this large $C$, $\Keq + 10^6$ steps of iteration could be not enough for producing a satisfactory output.

\paragraph{Fixed $\calR_k$}
The second part is to check for fixed $\calR_{k} = \calR$, \eg $\calR = 10^5$, how many more steps are needed after $\Keq$. From \eqref{eq:Rk-C}, simple calculation yields
\[
k - \Keq =  \sfrac{21(\sqrt{C} + 1)}{2} \sqrt[9]{\calR} - 20 .
\] 
Let again $C \in [10^4, 10^{12}]$, the value of $k - \Keq$ is shown in Figure~\ref{fig:fix-k-Keq}~(b). We can observe that when $C = 10^4$, only around $2,000$ steps are needed, while about $2\times 10^7$ steps are needed for $C = 10^{12}$.

\begin{figure}[!ht]
	\centering
	\subfloat[Value of $\calR_{k}$ when fix $k = \Keq + 10^6$]{ \includegraphics[width=0.375\linewidth]{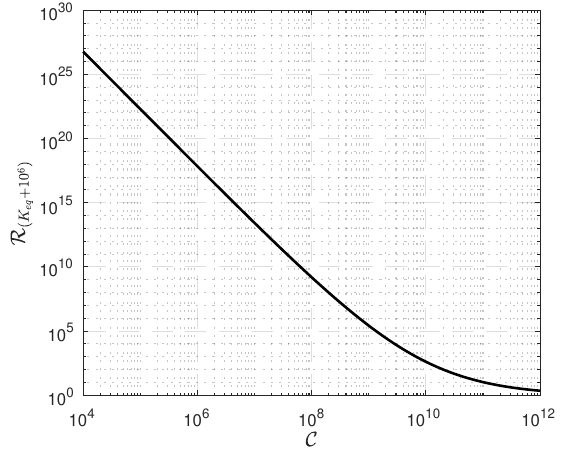} }  \hspace{24pt}
	\subfloat[Value of $k-\Keq$ when fix $\calR_k = 10^5$]{ \includegraphics[width=0.375\linewidth]{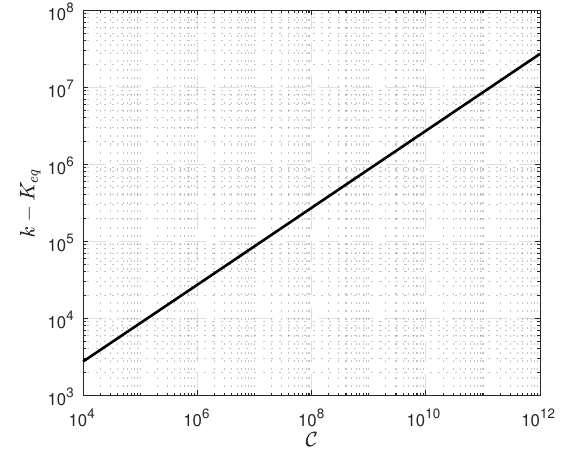} }  \\
	\caption{The dependence of $\calR_k$ on the iteration number $k$ and the condition number $C$. }
	\label{fig:fix-k-Keq}
\end{figure}

\begin{remark}
	When $C$ and $k$ are fixed, $\calR_k$ increases with $d$. This means if we consider only $\calR_k$, then the larger value of $d$ the better. However, one should not do so in practice, as larger $d$ will make the value of $\Keq$ much larger. As a result, proper choice of $d$ is a trade-off between $\Keq$ and $\calR_k$, which is the content of the next part.  
\end{remark}

\subsection{Continuous dynamical system perspective}

The above discussion implies the existence optimal choices of $d$. From continuous dynamical system perspective, we show that an optimal $d$ does indeed exists. What is interesting is that the optimal $d$ does not depend on condition number of the problem, but the accuracy of solution. The analysis is inspired by the result of \cite{su2014differential}.

\subsubsection{Optimal choice damping coefficient}

%
%
To prove the claim, we start from continuous dynamical system \eqref{eq:cds_x} first, showing that larger values of $\omega$ below leads to faster convergence, and then back to the discrete setting for the proposed claim.

For problem \eqref{eq:lse}, the associated continuous dynamical system reads: 
\beq\label{eq:cds_x}
    \ddot{x} + \sfrac{\omega}{t}\dot{x} + A^TA x = 0 ,
\eeq
where $\omega$ is the damping coefficient. 
Since $A^TA$ is symmetric, it can be diagonalised with invertible matrix $P$ and diagonal matrix $\Lambda=\mathrm{diag}(\lambda_1,\cdots,\lambda_n)$: $A^TA = P \Lambda P^{-1}$.
Let $y=P^{-1}x$, then we get
\beqn
    \ddot{y} + \sfrac{\omega}{t}\dot{y} + \Lambda y=0.
\eeqn
Since $\Lambda$ is diagonal, it is sufficient to consider each entry of $y$ that
\beqn
    \ddot{y}_i + \sfrac{\omega}{t}\dot{y}_i + \lambda_i y_i=0,\quad i=1,\cdots,n , 
\eeqn
where $n$ is the dimension of the problem. 
Let $\omega_i=\omega\lambda_i^{-1/2}$, $\nu_i=\frac{\omega_i-1}{2}$ and $z_i(t)=t^{\nu_i}y_i(\lambda_i^{-1/2}t)$ for $i=1,\cdots,n$. This change of variables results in Bessel's differential equations \cite{su2014differential}:
\beqn
    t^2\ddot{z}_i + t\dot{z}_i + (t^2-\nu_i^2)z_i=0,\quad i=1,\cdots,n ,
\eeqn
whose solution is
\beqn
    z_i=c_{i,1}J_{\nu_i} + c_{i,2}Y_{\nu_i},\quad i=1,\cdots,n,
\eeqn
where $J_{\nu_i}$ and $Y_{\nu_i}$ are the first and second kind of Bessel functions. Therefore, we get for $y_i$ that 
\beqn
\begin{aligned}
    y_i(\lambda_i^{-1/2}t)
    &= t^{-\nu_i}z_i(t)=t^{-\nu_i}\Pa{ c_{i,1}J_{\nu_i}(t) + c_{i,2}Y_{\nu_i}(t) } ,\\
    y_i(t)
    &= (\lambda_i^{1/2}t)^{-\nu_i}\Pa{ c_{i,1}J_{\nu_i}(\lambda_i^{1/2}t) + c_{i,2}Y_{\nu_i}(\lambda_i^{1/2}t)} .
\end{aligned}
\eeqn
For $J_{\nu_i}$ and $Y_{\nu_i}$, recall the following asymptotic forms of Bessel functions for positive and large argument $t$:
\[
\begin{aligned}
    J_\nu(t)
    =\sqrt{\sfrac{2}{\pi t}}\bPa{ \cos\Pa{ t-\sfrac{\nu\pi}{2}-\sfrac{\pi}{4} } + O(t^{-1}) } \qandq
    Y_\nu(t)
    =\sqrt{\sfrac{2}{\pi t}}\bPa{ \sin\Pa{ t-\sfrac{\nu\pi}{2}-\sfrac{\pi}{4} } + O(t^{-1}) }.
\end{aligned}
\]
As a result, 
\[
\begin{aligned}
    J_{\nu_i}(\lambda_i^{1/2}t)
    &\textstyle=\sqrt{\frac{2}{\pi\lambda_i^{1/2} t}}\bPa{ \cos\Pa{ \lambda_i^{1/2}t-\frac{(\omega\lambda_i^{-1/2}-1)\pi}{4}-\frac{\pi}{4} } + O(t^{-1}) } 
    \textstyle=\sqrt{\frac{2}{\pi\lambda_i^{1/2} t}}\bPa{ \cos\Pa{ \lambda_i^{1/2}t-\frac{\omega\lambda_i^{-1/2}\pi}{4} } + O(t^{-1}) },\\
    Y_{\nu_i}(\lambda_i^{1/2}t)
    &\textstyle=\sqrt{\frac{2}{\pi\lambda_i^{1/2} t}}\bPa{ \sin\Pa{ \lambda_i^{1/2}t-\frac{\omega\lambda_i^{-1/2}\pi}{4} }  + O(t^{-1}) }.
\end{aligned}
\]
Eventually, we get for $y_i$ that 
\beq
    y_i(t)=\sqrt{c_{i,1}^2 + c_{i,2}^2}\sqrt{\sfrac{2}{\pi}}\lambda_i^{-\frac{\omega_i}{4}}t^{-\frac{\omega_i}{2}}\sin \Pa{ \lambda_i^{1/2}t-\sfrac{\omega\lambda_i^{-1/2}\pi}{4} + \theta_i } + O(t^{-1-\frac{\omega_i}{2}}),\label{eq..y_iAsymptotic}
\eeq
where $\theta_i=\arctan\frac{c_{i,1}}{c_{i,2}}$ depends on $c_{i,1}$ and $c_{i,2}$ which are determined by the initial condition.


From the above asymptotics, we conclude that, in the continuum case (\ie ODEs), the convergence is faster for larger $\omega$. However, in the discrete case, we have to also consider the numerical error.
We consider the following FISTA-CD scheme
\[
\begin{aligned}
    y_k
    &= x_k + \tfrac{k-1}{k + d}(x_k-x_{k-1}),\\
    x_{k + 1}
    &=y_{k}-\gamma\nabla F(y_k),
\end{aligned}
\]
where $d=\omega-1$. Note that $x_k\approx x(k\tau)$ with step-size $\tau=\sqrt{\gamma}$. The algorithm is then rewritten as
\beqn
  \sfrac{x_{k + 1}-x_k}{\tau} = \sfrac{k-1}{k + d}\sfrac{x_k-x-_{k-1}}{\tau}-\tau\nabla F(y_k).
\eeqn
By Taylor expansion in $\tau$, we have
\[
\begin{aligned}
    \dot{x}(t) + \sfrac{1}{2}\ddot{x}(t)\tau + o(\tau)
    &= \sfrac{t-\tau}{t + d\tau}\Pa{ \dot{x}(t)-\sfrac{1}{2}\ddot{x}(t)\tau + o(\tau) } - \tau\nabla F(x(t)) + o(\tau)\\
    &= (1-\tfrac{\omega\tau}{t})\Pa{ \dot{x}(t)-\sfrac{1}{2}\ddot{x}(t)\tau + o(\tau) } - \tau\nabla F(x(t)) + o(\tau).
\end{aligned}
\]
Note that in the last step we have applied expansion
\beq
    \sfrac{t-\tau}{t + d\tau}=1-\sfrac{\omega\tau}{t + (\omega-1)\tau}=1-\sfrac{\omega\tau}{t} + \sfrac{\omega(\omega-1)\tau^2}{t^2} + \cdots.
\eeq
This makes sense only for $\frac{(\omega-1)\tau}{t}<1$. More precisely, the numerical error at time $T$ is $\epsilon_\text{num}=\frac{\omega\tau}{T}$.

{
By approximation \eqref{eq..y_iAsymptotic}, the truncation error (tolerance) is $\epsilon=|x(T)-x( + \infty)|=|x(T)|=\lambda_1^{-\frac{\bar{\omega}}{4}}T^{-\frac{\bar{\omega}}{2}}$ where $\bar{\omega}=\max_{1\leq i\leq n}\{\omega_i\}$. 
Thus $T^{-1}=\epsilon^{\frac{2}{\bar{\omega}}}\lambda_1^{\frac{1}{2}}$ and $\epsilon_\text{num}=\tau\lambda_1^{\frac{1}{2}} {\omega}\epsilon^{\frac{2}{\bar{\omega}}}$. We need to minimize 
\[
\textstyle \log \epsilon_\text{num}
= { \log(\tau\lambda_1^{1/2}) + \log{\omega} + \frac{2}{\bar{\omega}}\log\epsilon }
= { \log(\tau\lambda_1^{1/2}) + \log{\omega} + \frac{2\lambda_1^{1/2}}{\bar{\omega}}\log\epsilon } ,\enskip \bar{\omega} \geq 3 ,
\]
which leads to $0 =  \frac{1}{\bar{\omega}} + \frac{2\lambda_1^{1/2}}{\bar{\omega}^2}\log\epsilon$. 
As a result, the optimal choice of $\bar{\omega}$ is $\omega=-2\lambda_1^{1/2}\log\epsilon$, hence $-2\lambda_1^{1/2}\log\epsilon - 1$ for $d$. 
}

\subsubsection{Optimal lazy-start parameters}

Now we turn to the discrete case and discuss the optimal $d$, through the envelope $\calE_{d, k}$. 

\paragraph{Optimal $d$ for $\norm{\xk-\xsol}$}
We continue using problem \eqref{eq:lse}, with condition number $C = 2.735 \times 10^{8}$. 
Consider several different values of $d$ which are $d \in [5, 15, 25, 35, 45]$. The values of corresponding $\calE_{d, k}$ are plotted in Figure~\ref{fig:Ek-d}~(a). For each $k \in [1, 10^6]$, the minimum of $\calE_{d, k}$ is computed and plotted as a {red dotted line}. 

From Figure~\ref{fig:Ek-d}~(a), it can be observed that for each $d \in [5, 15, 25, 35, 45]$, their corresponding $\calE_{d, k}$ is the smallest for a certain range of $k$. For instance, for $d = 5$, $\calE_{5,k}$ is the smallest for $k$ between $1$ and about $1.75\times 10^5$. This verifies the result from continuous dynamical system that
\begin{itemize}
	\item There exists an optimal choice of $d$;
	\item The optimal $d$ depends on the accuracy of $\xk$. 
\end{itemize}
To illustrate, we consider the following test: under a given tolerance $\tol \in \ba{-2,...,-10}$, for each $d \in [2, 100]$ compute the minimal number of iterations, \ie $k$, needed such that
\[
\log(\calE_{d,k}) \leq \tol . 
\]
The obtained results are shown in Figure~\ref{fig:Ek-d}~(b), from where we can observe that for each $\tol \in \ba{-2,...,-10}$, the corresponding $k$ is a smooth curve that admits a minimal value $k^\star_{\tol}$ for optimal $d^\star_{\tol}$. 
The red line segment connects all the points of $(d^\star_{\tol}, k^\star_{\tol})$ which almost is a straight line. It indicates that one should choose small $d$ for high accuracy and increase the value for lower accuracy.

\begin{figure}[!ht]
	\centering
	\subfloat[{Comparison of $\calE_{d, k}$ for different $d$}]{ \includegraphics[width=0.36\linewidth]{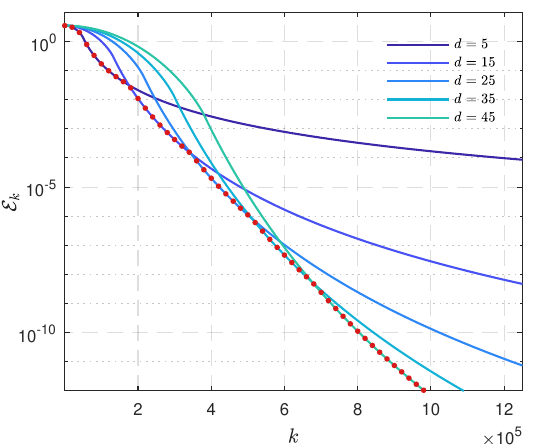} }  \hspace{24pt}
	\subfloat[{Value of $k$ for $\log(\calE_{d,k}) \leq \tol$}]{ \includegraphics[width=0.345\linewidth]{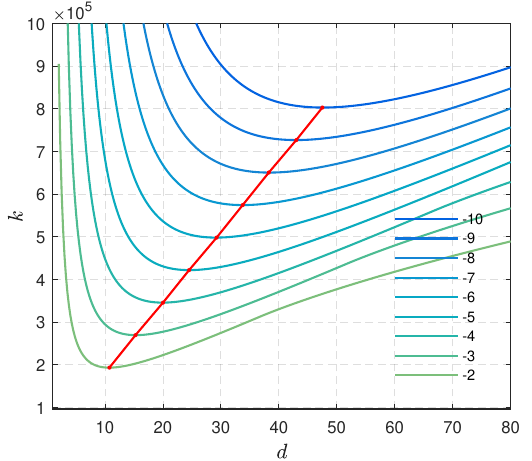} } \\[-1mm]
	\subfloat[{Optimal $d$ over $\tol$ under different $C$}]{ \includegraphics[width=0.347\linewidth]{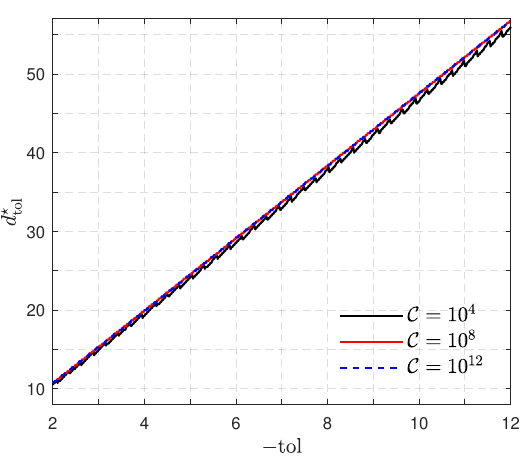} }  \hspace{24pt}
	\subfloat[{Difference between $\norm{\xk-\xsol}$ and $\norm{\xk-\xkm}$}]{ \includegraphics[width=0.36\linewidth]{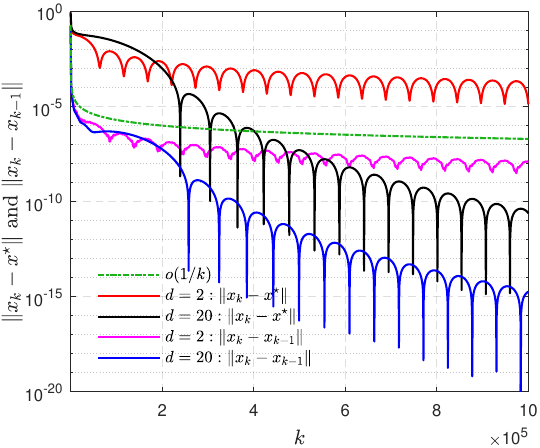} } 		\\
	\caption{Optimal choices of $d$ under different stopping tolerance.} 
	\label{fig:Ek-d}
\end{figure}

The red line in Figure~\ref{fig:Ek-d}~(b) accounts only for condition number $C = 2.735 \times 10^{8}$. In Figure~\ref{fig:Ek-d}~(c), we consider three different condition numbers $C \in \ba{10^4, 10^8, 10^{12}}$ and plot their corresponding optimal choices of $d$ under different $\tol$. Surprisingly, the obtained optimal choices for each $C$ are almost same, especially for $C = 10^8, 10^{12}$. From these three lines, we fit the following linear function
\[
d_{\tol}^\star = 10.75 + 4.6(-\tol-2)  ,
\]
which can be used to compute the optimal $d$ for a given stopping criterion on $\norm{\xk-\xsol}$.


\paragraph{Optimal $d$ for $\norm{\xk-\xkm}$}

To this point, we have presented detailed analysis on the advantage of lazy-start strategy. However, the analysis is conducted via the envelope $\calE_{d,k}$ of $\norm{\xk-\xsol}$ which requires the solution $\xsol$. While in practice, only $\norm{\xk-\xkm}$ is available, which makes the above discussion on optimal $d$ not practically useful. Therefore, we discuss briefly below on how to adapt the above result to $\norm{\xk-\xkm}$.

In Figure~\ref{fig:Ek-d} (d) we plot both $\norm{\xk-\xsol}$ and $\norm{\xk-\xkm}$ for the considered problem~\eqref{eq:lse} with $d=2$ and $d=20$. The red and magenta lines are for $d=2$ while the black and blue lines are for $d=20$. It can be observed that $\norm{\xk-\xkm}$ is several orders smaller than $\norm{\xk-\xsol}$, which is caused by the significant decay at the beginning of $\norm{\xk-\xkm}$, which is due to the fact that at beginning the convergence of $\norm{\xk-\xkm}$ is governed by the $o(1/k)$ rate established in Theorem~\ref{thm:rate-seq}; see the green dot-dash line. 

If we discard the beginning part of $\norm{\xk-\xkm}$, then the remainder can be seen as scaled $\norm{\xk-\xsol}$, \ie $\norm{\xk-\xkm} \approx \norm{\xk-\xsol}/10^s$ for some $s > 0$. Therefore, if some prior about this shift could be available, then the optimal choice of $d$ would be
\[
d_{\tol}^\star = 10.75 + 4.6(-\tol-2 - s) .
\]
For a given problem, in practice the value of $s$ can be estimated through the following strategy:
\begin{itemize}
	\item Run the FISTA iteration for sufficient number of iterations (\eg $3\times10^5$ steps in Figure \ref{fig:Ek-d}~(d)) and obtain a rough solution $\tilde{x}$ and also record the residual sequence $\norm{\xk-\xkm}$. 
	\item Rerun the iteration again (\eg for $10^5$ steps) and output the value of $\norm{\xk-\tilde{x}}$. Comparing $\norm{\xk-\xkm}$ and $\norm{\xk-\tilde{x}}$ one can then obtain an estimation of $s$. 
\end{itemize}
In practice, one can also simply choose $d \in [10, 80]$ which can provide consistent faster performance.

\begin{remark}$~$
	\begin{itemize}
		\item The discussion has been conducted through FISTA-CD, to extend the result to the case of FISTA-Mod, we may simply take $p = \frac{1}{d}$ and let $q \in ]0, 1]$. As we have seen from Figure~\ref{fig:tk-pqr}, the correspondence between FISTA-CD and FISTA-Mod is roughly $p = \frac{1}{d}$. 
		
		\item The discussion of this section considers only the least square problem \eqref{eq:lse} which is very simple. However, this does not mean that lazy-start strategy will fail for more complicated problems such as \eqref{eq:min-problem}, see Section~\ref{sec:experiment} for evidence of this. 
	\end{itemize}
\end{remark}

%% file: tex/sec-adaptive.tex
\section{Adaptive acceleration}\label{sec:adaptive}

We have discussed the advantages of the proposed FISTA-Mod scheme, particularly the lazy-start strategy. However, despite the advantage brought by lazy-start, FISTA-Mod and FISTA-CD still suffer the same drawback of FISTA-BT: the oscillation of $\Phi(\xk)-\Phi(\xsol)$ and $\norm{\xk-\xsol}$ as shown in Figure~\ref{fig:cmp-lse}. 
Therefore, in this section we discuss adaptive approaches to avoid oscillation. Note that here we only discuss adaptation to inertia, and refer to \cite{calatroni2017backtracking} for backtracking strategies for Lipschitz constant $L$.

The presented acceleration schemes cover two different cases: strong convexity is explicitly available, strong convexity is unknown (or $0$). For the first case, the optimal parameter choices are available. While for the latter, we need to adaptively estimate the (local) strong convexity. 

\subsection{Strong convexity is available}

For this case, we assume that $F$ of \eqref{eq:min-problem} is $\alpha$-strongly convex and $R$ is only convex, and derive the optimal setting of $p, q$ and $r$ for FISTA-Mod. 
Recall that under step-size $\gamma$, the optimal inertial parameter is $a^\star = \frac{1 - \sqrt{\gamma\alpha}}{1 + \sqrt{\gamma\alpha}}$. 
From~\eqref{eq:fista-pqr-ak} the limiting value of $\ak$, we have that for given $p,q \in ]0, 1]$, $r$ should be chosen such that
\[
\sfrac{2p+\sqrt{rp^2 + (4-r)q} - (4-r)}{2p+\sqrt{rp^2 + (4-r)q}} 
= \sfrac{1 - \sqrt{\gamma\alpha}}{1 + \sqrt{\gamma\alpha}}  .
\]
%
Solve the above equation we get the optimal choice of $r$ which reads
\beq\label{eq:opt-r}
\begin{aligned}
r = f(\alpha, \gamma; p,q) 
&\eqdef 4(1-p) + 4p \asol + (p^2 - q) (1-\asol)^2  \\
&= 4(1-p) + \sfrac{4p(1 - \sqrt{\gamma\alpha})}{1 + \sqrt{\gamma\alpha}}  + \sfrac{4\gamma\alpha(p^2 - q)}{ (1 + \sqrt{\gamma\alpha})^2} \leq 4 . 
\end{aligned}
\eeq
Note that we have $f(\alpha, \gamma; p,q) = 4$ for $\alpha = 0$, and $f(\alpha, \gamma; p,q) < 4$ for $\alpha >0$. 

Based on the above result, we propose below a generalization of FISTA-Mod which is able to adapt to the strong convexity of the problem to solve.

\begin{center}
	\begin{minipage}{0.95\linewidth}
		\begin{algorithm}[H]
			\caption{Strongly convex FISTA-Mod ({\bf $\alpha$-FISTA})} \label{alg:alpha-fista}
			{\noindent{\bf{Initial}}}: let $p, q > 0$ and $\gamma \leq 1/L$. For $\alpha \geq 0$, choose $r$ as $ r = f(\alpha, \gamma; p, q)$. Let $t_0 \geq 1$, {\hspace*{3.25em}}and $x_{0} \in \bbR^n, x_{-1} = x_{0}$. \\ 
			\Repeat{convergence}{ \vspace{-1em}
				\beq \label{eq:alpha-fista}
				\begin{gathered}
				\textstyle \tk = \frac{p + \sqrt{q + r\tkm^2}}{2}  , ~~~  \ak = \frac{\tkm-1}{\tk}  ,  \\
				\yk = \xk + \ak\pa{\xk-\xkm}    ,  \\
				\xkp = \prox_{\gamma R}\Pa{\yk - \gamma \nabla F(\yk)} . 
				\end{gathered}
				\eeq
			}
		\end{algorithm}
	\end{minipage}
\end{center}

\begin{remark}$~$
Since $f(\alpha, \gamma; p,q) = 4$ when $\alpha = 0$, the above algorithm mains the $o(1/k^2)$ convergence rate for non-strongly convex case, and in general we have the following convergence property for $\alpha$-FISTA,
		\[
		\Phi(\xk) - \Phi(\xsol)
		\leq \calC \min \Ba{ \tfrac{2L}{p^2(k+1)^2} , (1 - \sqrt{\gamma\alpha})^k } ,
		\]
		where $\calC>0$ is a constant.	

\end{remark}

\paragraph{Relation with \cite{calatroni2017backtracking}} 
Recently, combing FISTA scheme with strong convexity was studied in \cite{calatroni2017backtracking} where the authors also propose a generalization of FISTA scheme for strongly convex problems. They consider the case that $R$ is $\alpha_R$-strongly convex and $F$ is $\alpha_F$-strongly convex, and the whole problem is then $(\alpha = \alpha_R + \alpha_F)$-strongly convex. 
In \cite[Algorithm~1]{calatroni2017backtracking}, the following update rule of $\tk$ is considered
\beq\label{eq:luca}
\textstyle \tk = \frac{ 1 - q\tkm^2 + \sqrt{ (1-q\tkm^2)^2 + 4\tkm^2 } }{ 2 }	\qandq
\ak = \frac{\tkm-1}{\tk} \frac{1+\gamma\alpha_R-\tk\gamma\alpha}{1-\gamma\alpha_F}	,
\eeq
where $q = \frac{\gamma\alpha}{1+\gamma\alpha_R}$. 
As we shall see later in Section~\ref{sec:nesterov}, the above update rule is equivalent to Nesterov's optimal scheme \cite{nesterov2004introductory}; see also \cite{chambolle2016introduction} for discussions. 

When $\alpha > 0$, then \cite[Algorithm~1]{calatroni2017backtracking} achieves $O((1-\sqrt{q})^k)$ linear convergence rate. 
When $\alpha_R = 0, \alpha_F > 0$, we have $1-\sqrt{q} = 1-\sqrt{\gamma\alpha}$ which means \cite[Algorithm~1]{calatroni2017backtracking} and $\alpha$-FISTA achieves the same optimal rate. 
However, if both $\alpha_R > 0$ and $\alpha_F \geq 0$, then $1 - \sqrt{ \frac{\gamma\alpha}{1+\gamma\alpha_R} }
> 1 - \sqrt{\gamma\alpha} $, 
which means \eqref{eq:luca} achieves a sub-optimal convergence rate. 
%
As a matter of fact, if we transfer the strong convexity of $R$ to $F$, that is 
\[
R \eqdef R - \qfrac{\alpha_R}{2}\norm{x}^2	\qandq
F \eqdef F + \qfrac{\alpha_R}{2}\norm{x}^2	. 
\]
Then $R$ is convex and $F$ is $\alpha$-strongly convex, and the optimal rate would be $1-\sqrt{\gamma\alpha}$. 
Moreover, Moreover, redefining $R$ does not affect the complexity of computing $\prox_{\gamma R}$, as it is simply quadratic perturbation of proximity operator \cite[Lemma~2.6]{combettes2005signal}.

\subsection{Strong convexity is not available}

The goal of $\alpha$-FISTA is to avoid the oscillatory behavior of the FISTA schemes. In the literature, an efficient way to deal with oscillation is the restarting technique developed in \cite{o2012adaptive}. 
The basic idea of restarting is that, once the objective function value of $\Phi(\xk)$ is about to increase, the algorithm resets $t_k$ and $\yk$. 
Doing so, the algorithm achieves an almost monotonic convergence in terms of $\Phi(\xk) - \Phi(\xsol)$, and can be significantly faster than the original scheme; see \cite{o2012adaptive} or Section~\ref{sec:experiment} for detailed comparisons.

The strong convexity adaptive $\alpha$-FISTA (Algorithm~\ref{alg:alpha-fista}) considers only the situation where the strong convexity is explicitly available, which is very often not the case in practice. Moreover, the oscillatory behavior is independent of the strong convexity. As a consequence, an adaptive scheme is needed such that the following scenarios can be covered
\begin{itemize}
	\item $\Phi$ is \emph{globally} strongly convex with unknown modulus $\alpha$;
	\item $\Phi$ is \emph{locally} strongly convex with unknown modulus $\alpha$. 
	\item $\Phi$ is neither \emph{globally} nor \emph{locally} strongly convex;
\end{itemize}
%
Estimating the strong convexity in general is time consuming. 
Therefore, an efficient estimation approach is also needed. 
To address these problems, we propose a restarting adaptive scheme (Algorithm~\ref{alg:rada-fista}), which combines the restarting technique of \cite{o2012adaptive} and $\alpha$-FISTA.

\begin{center}
	\begin{minipage}{0.95\linewidth}
		\begin{algorithm}[H]
			\caption{Restarting and Adaptive $\alpha$-FISTA ({\bf Rada-FISTA})} \label{alg:rada-fista}
			{\noindent{\bf{Initial}}}: $p , q \in ]0, 1], r = 4$ and $\xi < 1$, $t_0 = 1, \gamma = 1/L$ and $x_{0} \in \calH, x_{-1} = x_{0}$. \\
			\Repeat{convergence}{
				$\bullet$~~\textrm{Run FISTA-Mod:} \vspace{-1.5em} \\
				\beqn
				\begin{gathered}
				\textstyle t_{k} = \frac{p + \sqrt{q + r t_{k-1}^2}}{2}  , ~~~  \ak = \frac{t_{k-1}-1}{t_{k}}  ,  \\
				\yk = \xk + \ak\pa{\xk-\xkm}    ,  \\
				\xkp = \prox_{\gamma R}\Pa{\yk - \gamma \nabla F(\yk)}   . 
				\end{gathered}  
				\eeqn
				$\bullet$~~\textrm{Restarting: if $(\yk - \xkp)^T(\xkp-\xk) \geq 0$,} \\
				\hspace{2.5em}{$\circ$}~~\textrm{Option I: $r = \xi r$ and $\yk = \xk$;}  \\
				\hspace{2.5em}{$\circ$}~~\textrm{Option II: $r = \xi r, t_k = 1$ and $\yk=\xk$.}
			}
		\end{algorithm}
	\end{minipage}
\end{center}

{\noindent}For the rest of the paper, we shall refer to Algorithm~\ref{alg:rada-fista} as ``Rada-FISTA''. Below, we provide some discussions:
\begin{itemize}
	\item 
	Compared to $\alpha$-FISTA, the main difference of Rada-FISTA is the restarting step which is originally proposed in \cite{o2012adaptive}. Such a strategy can successfully avoid the oscillatory behavior of $\Phi(\xk) - \Phi(\xsol)$. 
	\item 
	We provide two different options for the restarting step. In both options, we reset $\yk$ as in~\cite{o2012adaptive}. Meanwhile, we also rescale the value of $r$ by a factor $\xi$ which is strictly smaller than $1$. The purpose of rescaling is to approximate the optimal choice of $r$ in \eqref{eq:opt-r}.
	\item
	The difference between the two options is that $\tk$ is not reset to $1$ in ``Option I''. 
	Doing so, ``Option I'' will restart for more times than ``Option II'', however it will achieve faster practical performance; see Section~\ref{sec:experiment} the numerical experiments. It is worth noting that, for the restarting FISTA of \cite{o2012adaptive}, removing resetting $\tk$ could also lead to an acceleration. 
	
\end{itemize}
%

\subsection{Greedy FISTA}

We conclude this section by discussing how to further improve the performance of the restarting technique, achieving an even faster performance than Rada-FISTA and restarting FISTA \cite{o2012adaptive}.

The oscillation of FISTA schemes is caused by the fact that $\ak \to 1$. For the restarting scheme \cite{o2012adaptive}, resetting $\tk$ to $1$ forces $\ak$ to increase from $0$ again, become close enough to $1$ and cause the next oscillation, then the scheme restarts. 
With such a loop, if we can shorten the gap between two restarts, then maybe extra acceleration could be obtained. It turns out that using constant $\ak$ (close or equal to $1$) can achieve this goal. 
Therefore, we propose the following greedy restarting scheme.

\begin{center}
	\begin{minipage}{0.95\linewidth}
		\begin{algorithm}[H]
			\caption{Greedy FISTA} \label{alg:greedy} 
			{\noindent{\bf{Initial}}}: let $\gamma \in [\frac{1}{L}, \frac{2}{L}[$ and $\xi < 1, S > 1$, choose $x_{0} \in \bbR^n, x_{-1} = x_{0}$. \\ 
			\Repeat{convergence}{
				$\bullet$~~\textrm{Run the iteration:} \vspace{-1.5em} \\
				\beq \label{eq:greedy}
				\begin{aligned}
					\yk &= \xk + (\xk - \xkm) , \\ 
					\xkp &= \prox_{\gamma R}\Pa{\yk - \gamma \nabla F(\yk)} . 
				\end{aligned}
				\eeq
				$\bullet$~~\textrm{Restarting: if $(\yk - \xkp)^T(\xkp-\xk) \geq 0$, then $\yk = \xk$;} \\
				$\bullet$~~\textrm{Safeguard: if $\norm{\xkp-\xk} \geq S \norm{x_1-x_0}$, then $\gamma = \max\ba{\xi\gamma, \frac{1}{L}}$;} 
			}
		\end{algorithm}
	\end{minipage}
\end{center}

We abuse the notation by calling the above algorithm ``Greedy FISTA'', which uses constant inertial parameter $\ak \equiv 1$ for the momentum term:
\begin{itemize}
	\item A larger step-size (than $1/L$) is chosen for $\gamma$, which can further shorten the oscillation period;
	\item As such a large step-size may lead to divergence, we add a ``safeguard'' step to ensure the convergence. This step shrinkages the value of $\gamma$ when certain condition (\eg $\norm{\xkp-\xk} \geq S \norm{x_1-x_0}$) is satisfied. Eventually we will have $\gamma = 1/L$ if the safeguard is activated a sufficient number of times. 
\end{itemize}
In practice, we find that $\gamma \in [1/L, 1.3/L]$ provides faster performance than Rada-FISTA and restarting FISTA of \cite{o2012adaptive}; See Section~\ref{sec:experiment} for more detailed comparisons.

%% file: tex/sec-nesterov.tex

\begin{center}
	\begin{minipage}{0.95\linewidth}
		\begin{algorithm}[H]
			\caption{Accelerated proximal gradient (APG)} \label{alg:apg}
			{\noindent{\bf{Initial}}}: $\tau \in [0, 1],  \theta_0 = 1$, $\gamma = 1/L$ and $x_{0} \in \calH, x_{-1} = x_{0}$. \\
			\Repeat{convergence}{
				\textrm{Estimate the local strong convexity $\alpha_{k}$;} \vspace*{-1.0em} \\
				\beqn\label{eq:its-fpd}
				\begin{gathered}
				\theta_{k} ~~\textrm{solves}~~ \theta_{k}^2 =  (1-\theta_{k})\theta_{k-1}^2 + \tau \theta_{k}  , \enskip \textstyle \ak = \frac{\theta_{k-1}(1-\theta_{k-1})}{\theta_{k-1}^2+\theta_{k}}  ,  \\
				\yk = \xk + \ak\pa{\xk-\xkm}    ,  \\
				\xkp = \prox_{\gamma R}\Pa{\yk - \gamma \nabla F(\yk)} .
				\end{gathered}
				\eeqn
			}
		\end{algorithm}
	\end{minipage}
\end{center}

\section{Nesterov's accelerated scheme}\label{sec:nesterov}

In this section, we turn to Nesterov's accelerated gradient method \cite{nesterov2004introductory} and extend the above results to this scheme. In the book \cite{nesterov2004introductory}, Nesterov introduces several different acceleration schemes, in the following we mainly focus on the ``Constant Step Scheme, III''. Applying this scheme to solve \eqref{eq:min-problem}, we obtain the accelerated proximal gradient method (APG) described in Algorithm \ref{alg:apg}.

When the problem \eqref{eq:min-problem} is $\alpha$-strongly convex, then by setting $\tau = \sqrt{\alpha/L}$ and $\theta_0 = \tau$, we have
\[
\theta_k \equiv \tau \qandq
\ak \equiv \sfrac{ 1 - \sqrt{\gamma\alpha} }{ 1 + \sqrt{\gamma\alpha} }  ,
\] 
and the iterate achieves the optimal linear convergence speed, \ie $1 - \sqrt{\gamma\alpha}$, as we have already discussed in the previous sections. In the rest of this section, we first build connections between the parameter updates of APG with $\alpha$-FISTA, and then extend the lazy-start strategy to APG.

\subsection{Connection with $\alpha$-FISTA}\label{sec:connections}

Consider the following equation of $\theta$ parametrised by $0\leq \tau \leq \sigma \leq 1$, which recovers the $\theta_k$ update of APG for $\sigma = 1$,
\beq\label{eq:poly-theta}
\theta^2 + (\sigma\theta_{k-1}^2 - \tau) \theta - \theta_{k-1}^2 = 0 .
\eeq
The definition of $\ak$ implies $\theta_k \in [0, 1]$ for all $k\geq1$. Therefore, the $\theta_{k}$ we seek from above \eqref{eq:poly-theta} reads 
\beq\label{eq:thetak-APG-pq}
\textstyle \theta_{k} 
= \frac{ - (\sigma\theta_{k-1}^2-\tau) + \sqrt{(\sigma\theta_{k-1}^2-\tau)^2 + 4\theta_{k-1}^2 }}{2}	 .
\eeq
It is then easy to verify that $\theta_k$ is convergent and $\lim_{k\to+\infty}\theta_k = \sqrt{ \frac{\tau}{\sigma} }$.  
Back to \eqref{eq:thetak-APG-pq}, we have
\[
\textstyle \theta_{k} 
= \frac{ 2\theta_{k-1}^2}{ (\sigma\theta_{k-1}^2-\tau) + \sqrt{(\sigma\theta_{k-1}^2-\tau)^2 + 4\theta_{k-1}^2 } }  
= \frac{ 2}{ (\sigma-\tau/\theta_{k-1}^2) + \sqrt{(\sigma-\tau/\theta_{k-1}^2)^2 + 4 } }  .
\]
Letting $\tk = 1/\theta_{k}$ and substituting back to the above equation lead to
\beq\label{eq:tk-apg}
\textstyle \tk = \frac{ (\sigma-\tau\tkm^2) + \sqrt{(\sigma-\tau\tkm^2)^2 + 4\tkm^2 } }{ 2 }  .
\eeq
Note that the update rule \eqref{eq:luca} of \cite{calatroni2017backtracking} is a special case of above equation with $\sigma = 1$ and $\tau = \frac{\gamma\alpha}{1+\gamma\alpha_R}$. 
Moreover,
\beqn
\tk \to 
\left\{
\begin{aligned}
+\infty &: \tau = 0	,		\\
\sqrt{ \sfrac{\sigma}{\tau} }	&: \tau \in ]0, 1] .
\end{aligned}
\right.
\eeqn
Depending on the choices of $\sigma, \tau$, we have
\begin{itemize}
	\item When $(\sigma, \tau) = (1, 0)$, APG is equivalent to the original FISTA-BT scheme;
	\item When $(\sigma, \tau) = (1, {\gamma\alpha}{})$, APG is equivalent to \cite[Algorithm~1]{calatroni2017backtracking} for adapting to strong convexity. 
\end{itemize}
Building upon the above connection, we can extend the previous result of FISTA-Mod to the case of APG.


\subsection{A modified APG}

Extending the FISTA-Mod and $\alpha$-FISTA to the case of APG, we propose the following modified APG scheme which we name as ``APG-Mod''.

\begin{center}
	\begin{minipage}{0.95\linewidth}
		\begin{algorithm}[H]
			\caption{A modified APG scheme (\bf{APG-Mod})} \label{alg:mAPG}
			{\noindent{\bf{Initial}}}: Let $\sigma \in [0, 1], \gamma = 1/L$ and $\tau = \gamma\alpha\sigma, \theta_0 \in [0, 1]$. Set $x_{0} \in \calH, x_{-1} = x_{0}$. \\
			\Repeat{convergence}{
				\beq\label{eq:mapg}
				\begin{gathered}
				\theta_{k} ~~\textrm{solves}~~ \theta_k^2  =  (1- \sigma \theta_{k})\theta_{k-1}^2 + \tau \theta_k  , \\
				\textstyle \ak = \frac{\theta_{k-1}(1-\theta_{k-1})}{\theta_{k-1}^2+\theta_{k}}  ,  \\
				\yk = \xk + \ak\pa{\xk-\xkm}    ,  \\
				\xkp = \prox_{\gamma R}\Pa{\yk - \gamma \nabla F(\yk)} .
				\end{gathered}
				\eeq
			}
		\end{algorithm}
	\end{minipage}
\end{center}

\paragraph{Non-strongly convex case}
For the case $\Phi$ is only convex, we have $\tau = 0$, then $\theta_k$ is the root of the equation
\[
\theta^2 + \sigma \theta_{k-1}^2 \theta - \theta_{k-1}^2 = 0	.
\]
Owing to Section~\ref{sec:connections}, we have that APG-Mod is equivalent to FISTA-Mod with $p = \sigma$ and $q = \sigma^2$. Therefore, we have the following convergence result for APG-Mod which is an extension of Theorems~\ref{thm:rate-obj} and~\ref{thm:rate-seq}.

\begin{corollary} 
	
	For APG-Mod scheme Algorithm~\ref{alg:mAPG}, let $\tau = 0$ and $\sigma \in ]0, 1]$, then
	\begin{itemize}
		\item For the objective function value, 
		\[
		\Phi(\xk) - \Phi(\xsol) \leq   \sfrac{2L}{\sigma^2(k+1)^2} \norm{x_{0} - \xsol}^2  .
		\]
		If moreover $\sigma < 1$, we have $\Phi(\xk) - \Phi(\xsol) = o(1/k^2)$. 
		\item Let $\sigma < 1$, then there exists an $\xsol \in \Argmin(\Phi)$ to which the sequence $\seq{\xk}$ converges weakly and $\norm{\xk-\xkm} = o(1/k)$.  
	\end{itemize}
	
\end{corollary}

\begin{remark}
	Given the correspondence between $\sigma $ of APG-Mod and $p$ of FISTA-Mod, owing to Proposition~\ref{prop:lazy-start}, we obtain the lazy-start APG-Mod by choosing $\sigma \in [\frac{1}{80}, \frac{1}{10}]$. 
\end{remark}

\paragraph{Strongly convex case}

When the problem \eqref{eq:min-problem} is strongly convex with modulus $\alpha >0$, as $\tau = {\gamma\alpha\sigma}$, then according to Section~\ref{sec:connections}, we have
\[
\theta_k \to \sqrt{ \qfrac{\tau}{\sigma} } = \sqrt{\gamma\alpha} \qandq
\ak \to \sfrac{ 1 - \sqrt{\gamma\alpha} }{ 1 + \sqrt{\gamma\alpha} }	,
\]
which means that APG-Mod achieves the optimal convergence rate $1 - \sqrt{\gamma\alpha}$.

\begin{remark}
We can also extend the Rada-FISTA to APG, we shall forgo the details here as it is rather trivial. 
\end{remark}

%% file: tex/sec-experiment.tex

\section{Numerical experiments}\label{sec:experiment}

Now we present numerical experiments on problems arising from inverse problems, signal/image processing, machine learning and computer vision to demonstrate the performance of the proposed schemes. 
Throughout this section, the following schemes and corresponding settings are considered:
\begin{itemize}
	\item The original FISTA-BT scheme \cite{fista2009};
	\item The proposed FISTA-Mod (Algorithm~\ref{alg:fista-mod}) with $p = 1/20$ and $q = 1/2$, \ie the lazy-start strategy;
	\item The restarting FISTA of \cite{o2012adaptive};
	\item The Rada-FISTA scheme (Algorithm~\ref{alg:rada-fista});
	\item The greedy FISTA  (Algorithm~\ref{alg:greedy}) with $\gamma = 1.3/L, S = 1$ and $\xi = 0.96$.
\end{itemize}
The $\alpha$-FISTA (Algorithm~\ref{alg:alpha-fista}) is not considered here, except in Section~\ref{sec:lse-continue},  since most of the problems considered are only locally strongly convex along certain direction \cite{liang2017activity}. 
The corresponding MATLAB source code for reproducing the experiments is available at: \url{https://github.com/jliang993/Faster-FISTA}.

All the schemes are running with same initial point, which is $x_0 = \mathbf{1}\times 10^{4}$ for the least square problem and $x_0 = \mathbf{0}$ for all other problems. 
In terms of comparison criterion, we mainly focus on $\norm{\xk-\xsol}$ where $\xsol \in \Argmin(\Phi)$ is a global minimizer of the optimization problem.

\subsection{Least square \eqref{eq:lse} continue}\label{sec:lse-continue}

First we continue with the least square estimation \eqref{eq:lse} discussed in Section~\ref{sec:lazy}, and present a comparison of different schemes in terms of both $\norm{\xk-\xsol}$ and $\Phi(\xk)-\Phi(\xsol)$. Since this problem is strongly convex, the optimal scheme (\ie $\alpha$-FISTA) is also considered for comparison.

The obtained results are shown in Figure~\ref{fig:cmp-lse-continue}, with $\norm{\xk-\xsol}$ on the left and $\Phi(\xk)-\Phi(\xsol)$ on the right. From these comparisons, we obtain the following observations: 
\begin{itemize}
	\item FISTA-BT is faster than FISTA-Mod for $k \leq 3\times 10^{5}$, and becomes increasing slow afterwards. This agrees with our discussion in Figure~\ref{fig:Ek-d} that each parameter choice (of $p$ and $q$, and $d$ for FISTA-CD) is the fastest for a certain accuracy;
	\item 
	$\alpha$-FISTA is the only scheme whose performance is monotonic in terms of both $\norm{\xk-\xsol}$ and $\Phi(\xk)-\Phi(\xsol)$. It is also faster than both FISTA-BT and FISTA-Mod;
	\item
	The three restarting adaptive schemes are the fastest among tested schemes, with Greedy FISTA being faster than the other two. 
\end{itemize}

\begin{figure}[!ht]
\centering
\subfloat[$\norm{\xk-\xsol}$]{ \includegraphics[width=0.375\linewidth]{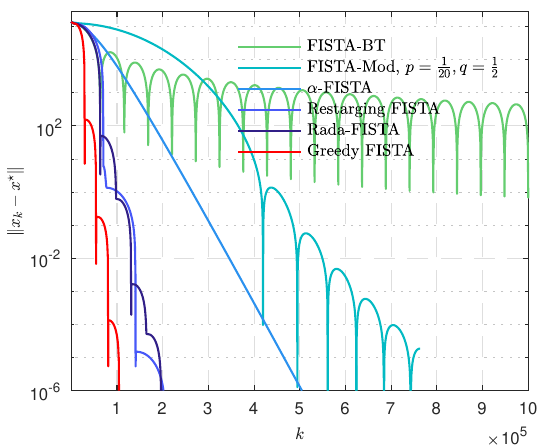} }  \hspace{24pt}
\subfloat[$\Phi(\xk)-\Phi(\xsol)$]{ \includegraphics[width=0.375\linewidth]{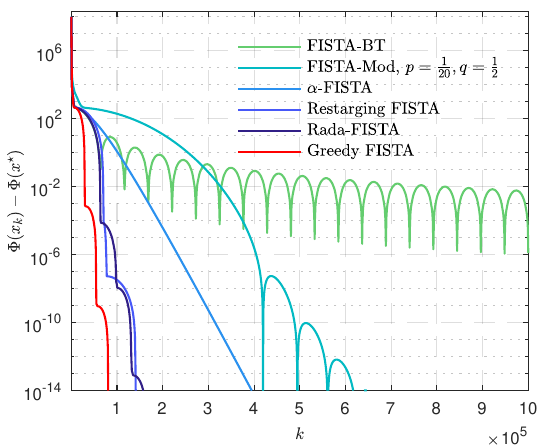} }  \\
\caption{Comparison of different FISTA schemes for least square problem \eqref{eq:lse}. }
\label{fig:cmp-lse-continue}
\end{figure}

\subsection{Linear inverse problem and regression problems}

From now on, we turn to dealing with problems that are only locally strongly convex around the solution of the problem. We refer to \cite{liang2017activity} for a detailed characterization of such local neighborhoods.

\paragraph{Linear inverse problem} 
Consider the following regularised least square problem
\beq \label{eq:P-lambda}
\min_{x \in \bbR^n} \mu R(x)  + \sfrac{1}{2} \norm{\calK x - f}^2  ,
\eeq
where $\mu > 0$ is trade-off parameter, $R$ is the regularization function. The forward model of \eqref{eq:P-lambda} reads
\beq\label{eq:observation}
f = \calK \xob + w   ,
\eeq
where $\xob \in \bbR^n$ is the original object that obeys certain prior (\eg sparsity and piece-wise constant), $f \in \bbR^m$ is the observed data, $\calK : \bbR^n \to \bbR^m$ is some linear operator, and $w \in \bbR^m$ stands for noise. 
In the experiments, we consider $R$ being $\ell_{\infty}$-norm and total variation \cite{rudin1992nonlinear}.  Here $\calK$ is generated from the standard Gaussian ensemble and the following setting is considered:
\begin{description}[leftmargin=3.5cm]
\item[{$\ell_{\infty}$-norm}] $(m, n) = (1020, 1024)$, $\xob$ has $32$ saturated entries;
\item[{Total variation}] $(m, n) = (256, 1024)$, $\nabla \xob$ is $32$-sparse.
\end{description}

\paragraph{Sparse logistic regression}
A sparse logistic regression problem for binary classification is also considered. Let $(h_i, l_i) \in \bbR^n \times \{\pm 1\} ,~ i=1,\cdots,m$ be the training set, where $h_{i} \in \bbR^n$ is the feature vector of each data sample, and $l_{i}$ is the binary label. 
%
The formulation of sparse logistic regression reads
\beq\label{eq:lr}
\min_{x \in \bbR^n }  \mu \norm{x}_{1} + \sfrac{1}{m} \msum_{i=1}^m \log\Pa{ 1+e^{ -l_{i} h_{i}^T x } }  . 
\eeq
%
The \texttt{australian} data set from LIBSVM\footnote{\url{https://www.csie.ntu.edu.tw/~cjlin/libsvmtools/datasets/}} is considered.

\begin{figure}[!ht]
	\centering
	\subfloat[$\ell_{\infty}$-norm]{ \includegraphics[width=0.32\linewidth]{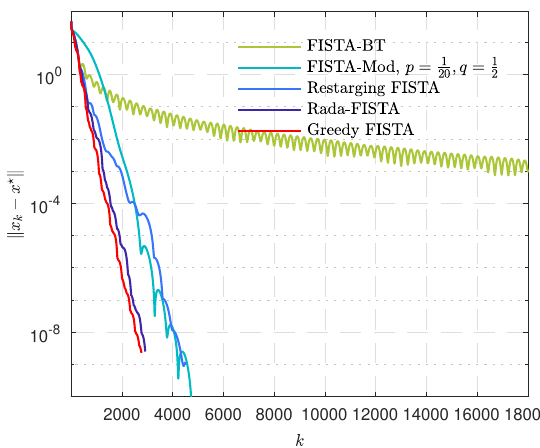} }  
	\subfloat[Total variation]{ \includegraphics[width=0.32\linewidth]{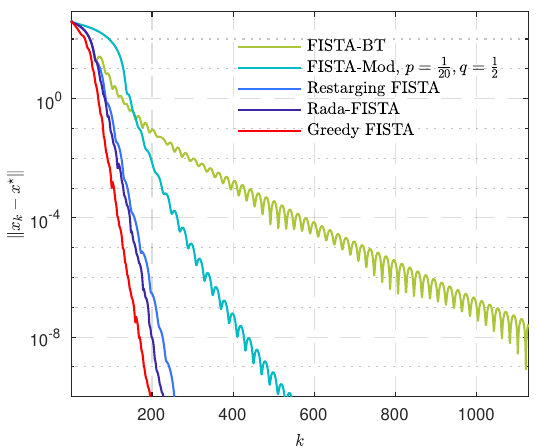} }  
	\subfloat[Sparse logistic regression]{ \includegraphics[width=0.32\linewidth]{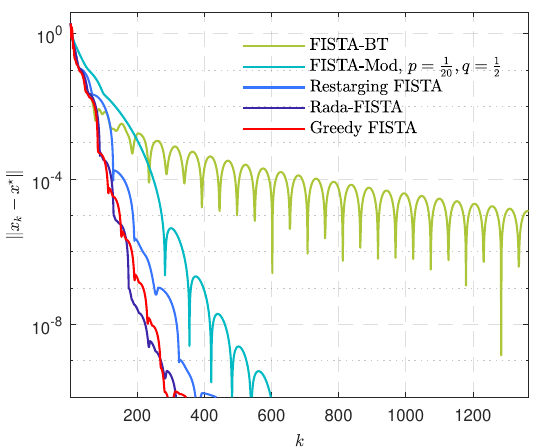} }  \\
	\caption{Comparison of different FISTA schemes for linear inverse problems and sparse logistic regression.}
	\label{fig:cmp-ip-slr}
\end{figure}

The observations are shown in Figure~\ref{fig:cmp-ip-slr}. 
Although these problems are only locally strongly convex around the solution, the observations are quite close to those of least square problem discussed above:
\begin{itemize}
	\item The lazy-start FISTA-Mod is slower than FISTA-BT at the beginning, and eventually becomes much faster, as predicted. For the $\ell_{\infty}$-norm, it is more than $10$ times faster if we need the precision to be $\norm{\xk-\xsol} \leq 10^{-10}$;
	\item The restarting adaptive schemes are the fastest ones, and the Greedy FISTA is the fastest of all.
\end{itemize}

\subsection{Principal component pursuit}

Lastly, we consider the principal component pursuit (PCP) problem \cite{candes2011robust}, and apply it to decompose a video sequence into background and foreground.

Assume that a real matrix $f \in\bbR^{m\times n}$ can be written as
\[
f = \xobl + \xobs + w,
\]
where $\xobl$ is low--rank, $\xobs$ is sparse and $w$ is the noise. 
The PCP proposed in \cite{candes2011robust} attempts to recover/approximate $(\xobl,\xobs)$ by solving the following convex optimization problem
\beq
\label{eq:rpca}
\min_{\xl, \xs \in\bbR^{m\times n}}~ \sfrac{1}{2}\norm{f-\xl-\xs}_{F}^2 + \mu \norm{\xs}_1 + \nu \norm{\xl}_*  ,
\eeq
where $\norm{\cdot}_F$ is the Frobenius norm. 
Observe that for fixed $\xl$, the minimizer of \eqref{eq:rpca} is $\xs^\star=\prox_{\mu{\norm{\cdot}_1}}(f - \xl)$. Thus, \eqref{eq:rpca} is equivalent to
\beq
\label{eq:rpcame}
\min_{\xl\in\bbR^{m\times n}}~^1\Pa{\mu \norm{\cdot}_1}(f-\xl) + \nu \norm{\xl}_*  ,
\eeq
where $^1\Pa{{\mu \norm{\cdot}_1}}(f-\xl) = \min_{z} \frac{1}{2}\norm{f-\xl-z}_F^2+ \mu\norm{z}_1$ is the Moreau Envelope of $\mu {\norm{\cdot}_1}$ of index $1$, and hence has $1$-Lipschitz continuous gradient.

We use the video sequence from \cite{li2004statistical} and the obtained result is demonstrated in Figure \ref{fig:cmp-pcp}. Again, we obtain consistent observations with the above examples. Moreover, the performance of lazy-start FISTA-Mod is very close to the restarting adaptive schemes.

\begin{figure}[!ht]
\centering
\subfloat[Original frame]{ \includegraphics[width=0.225\linewidth]{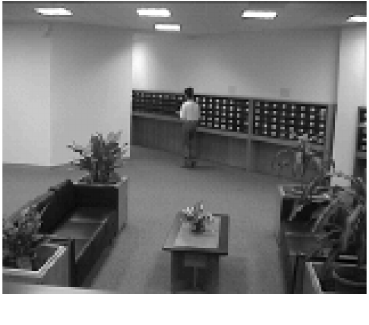} }  
\subfloat[Sparse component]{ \includegraphics[width=0.225\linewidth]{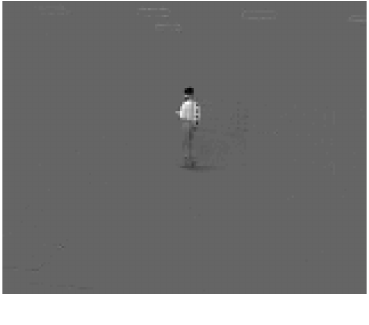} }  
\subfloat[Low-rank component]{ \includegraphics[width=0.225\linewidth]{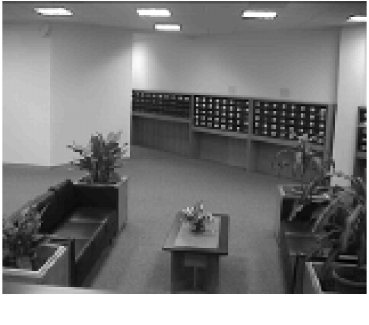} }  
\subfloat[Performance comparison]{ \includegraphics[width=0.245\linewidth]{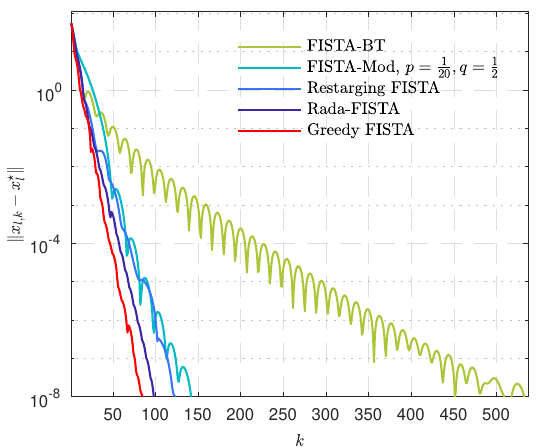} }  \\
\caption{Comparison of different FISTA schemes for principal component pursuit problem. (a) original frame; (b) foreground; (c) background; (d) performance comparison. }
\label{fig:cmp-pcp}
\end{figure}

In all these experiments we find that the proposed variants can perform better than the original versions but restarting are consistently faster. Greedy FISTA was the best in every example shown.